\documentclass[UTF-8,reqno]{amsart}
\usepackage{enumerate}
\usepackage{amssymb,url,color, booktabs}

\usepackage{mathrsfs}
\usepackage{amsmath}

\usepackage{fancyhdr}
\usepackage[nobysame]{amsrefs}
\BibSpec{article}{%
+{}{\PrintAuthors} {author}
+{,}{ \textrm} {title}
+{.}{ \textit} {journal}
+{,}{ \textbf} {volume}
+{}{ \parenthesize} {date}
+{,}{ } {pages}
+{.}{ arXiv:} {eprint}
+{.}{} {transition}
}
\BibSpec{book}{%
+{}{\PrintAuthors} {author}
+{,}{ \textit} {title}
+{.}{ \textrm} {series} 
+{,}{ Vol.} {volume}
+{.}{ } {publisher}
+{,}{ } {date}
+{.}{} {transition}
}
\usepackage{color}
\usepackage[colorlinks=true]{hyperref}
\hypersetup{
    linkcolor=blue,          
    citecolor=red,        
    filecolor=blue,      
    urlcolor=cyan
}

\numberwithin{equation}{section}

\newcommand{\be}{\begin{eqnarray}}
\newcommand{\ee}{\end{eqnarray}}
\newcommand{\ce}{\begin{eqnarray*}}
\newcommand{\de}{\end{eqnarray*}}
\newtheorem{theorem}{Theorem}[section]
\newtheorem{lemma}[theorem]{Lemma}
\newtheorem{remark}[theorem]{Remark}
\newtheorem{definition}[theorem]{Definition}
\newtheorem{proposition}[theorem]{Proposition}
\newtheorem{Examples}[theorem]{Example}
\newtheorem{corollary}[theorem]{Corollary}
\newtheorem{assumption}[theorem]{Assumption}

\newenvironment{proof of theorem 1.2 and 1.3}{{\it Proof of Theorem 1.2 and 1.3}.}{{\hfill 	
$\square$\hskip - \parfillskip}}
\newenvironment{proof of theorem 1.4}{{\it Proof of Theorem 1.4}.}{{\hfill 	
		$\square$\hskip - \parfillskip}}
\newenvironment{proof of theorem 1.5}{{\it Proof of Theorem 1.5}.}{{\hfill 	
		$\square$\hskip - \parfillskip}}
\newenvironment{proof of theorem 1.6}{{\it Proof of Theorem 1.6}.}{{\hfill 	
		$\square$\hskip - \parfillskip}}

\makeatletter
\newcommand{\rmnum}[1]{\romannumeral #1}
\newcommand{\Rmnum}[1]{\expandafter\@slowromancap\romannumeral #1@}
\makeatother

\def\eps{\varepsilon}

\def\e{\mathrm{e}}

\def\a{\alpha}

\def\p{\partial}

\def\g{\gamma}

\def\[{{\Big[}}
\def\]{{\Big]}}
\def\<{{\langle}}
\def\>{{\rangle}}
\def\({{\Big(}}
\def\){{\Big)}}

\def\bx{{\mathbf{x}}}

\def\min{{\mathord{{\rm min}}}}

\def\={&\!\!=\!\!&}

\def\cL{{\mathcal L}}

\def\mH{{\mathbb H}}

\def\mK{{\mathbb K}}

\def\mR{{\mathbb R}}
\def\mS{{\mathbb S}}

\def\1{{\mathbf{1}}}

\def\geq{\geqslant}
\def\leq{\leqslant}
\def\ge{\geqslant}
\def\le{\leqslant}

\def\k{\kappa}

\def\eps{\varepsilon}

\def\e{\mathrm{e}}

\def\a{\alpha}

\def\p{\partial}

\def\g{\gamma}

\def\[{{\Big[}}
\def\]{{\Big]}}
\def\<{{\langle}}
\def\>{{\rangle}}
\def\({{\Big(}}
\def\){{\Big)}}

\def\bx{{\mathbf{x}}}

\def\min{{\mathord{{\rm min}}}}

\def\={&\!\!=\!\!&}
\def\bt{\begin{theorem}}
\def\et{\end{theorem}}
\def\bl{\begin{lemma}}
\def\el{\end{lemma}}
\def\br{\begin{remark}}
\def\er{\end{remark}}
\def\bx{\begin{Examples}}
\def\ex{\end{Examples}}
\def\bd{\begin{definition}}
\def\ed{\end{definition}}
\def\bp{\begin{proposition}}
\def\ep{\end{proposition}}
\def\bc{\begin{corollary}}
\def\ec{\end{corollary}}

\def\geq{\geqslant}
\def\leq{\leqslant}
\def\ge{\geqslant}
\def\le{\leqslant}

 \def\nn{\nabla}

\def\<{\langle} \def\>{\rangle}

\def\bpf{\begin{proof}}
\def\epf{\end{proof}}

\allowdisplaybreaks

\begin{document}
	
\title{A Class Of Curvature Flows Expanded By Support Function And Curvature Function in the Euclidean space and Hyperbolic space}\thanks{\it {This research was partially supported by NSFC (Nos. 11761080 and 11871053).}}
\author{Shanwei Ding and Guanghan Li}


\address{School of Mathematics and Statistics, Wuhan University, Wuhan 430072, China.
}

\begin{abstract}
In this paper, we first consider a class of expanding flows of closed, smooth, star-shaped hypersurface in Euclidean space $\mathbb{R}^{n+1}$ with speed $u^\alpha f^{-\beta}$, where $u$ is the support function of the hypersurface, $f$ is a smooth, symmetric, homogenous of degree one, positive function of the principal curvatures of the hypersurface on a convex cone. For $\alpha \le 0<\beta\le 1-\alpha$, we prove that the flow has a unique smooth solution for all time, and converges smoothly after normalization, to a sphere centered at the origin. In particular, the results of Gerhardt \cite{GC3} and Urbas \cite{UJ2} can be recovered by putting $\a=0$ and $\beta=1$ in our first result. If the initial hypersurface is convex, this is our previous work \cite{DL}. If $\alpha \le 0<\beta< 1-\alpha$ and the ambient space is hyperbolic space $\mathbb{H}^{n+1}$, we prove that the flow $\frac{\p X}{\p t}=(u^\alpha f^{-\beta}-\eta u)\nu$ has a longtime existence and smooth convergence to a coordinate slice. The flow in $\mathbb{H}^{n+1}$ is equivalent (up to an isomorphism) to a re-parametrization of the original flow in $\mathbb{R}^{n+1}$ case. Finally, we find a family of monotone quantities along the flows in $\mathbb{R}^{n+1}$. As applications, we give a new proof of a family of inequalities involving the weighted integral of $k$th elementary symmetric function for $k$-convex, star-shaped hypersurfaces, which is an extension of the quermassintegral inequalities in \cite{GL2}.

\end{abstract}

\maketitle
\setcounter{tocdepth}{2}
\tableofcontents

\section{Introduction}
Flows of convex hypersurfaces  by a class of speed functions which are homogenous and symmetric in principal curvatures have been extensively studied in the past four decades. Well-known examples include the mean curvature flow \cite{HG}, and the Gauss curvature flow \cite{BS,FWJ}. In \cite{HG} Huisken showed that the flow has a unique smooth solution and the hypersurface converges to a round sphere if the initial hypersurface is closed and convex. Later, a range of flows with the speed of homogenous of degree one in principal curvatures were established,  see \cite{B0,B1,CB1,CB2} and references therein.

For star-shaped hypersurface $M_0$, Gerhardt \cite{GC3,GC} and Urbas \cite{UJ2} studied the flow with concave curvature function $f$ which satisfies $f\vert_{\Gamma}>0$ and $f\vert_{\p\Gamma}=0$ for an open convex symmetric cone $\Gamma$ containing the positive cone $\Gamma_+$, and proved a similar convergence result. Scheuer \cite{SJ} improved the asymptotical behavior of the flow considered in \cite{GC} by showing that the flow becomes close to a flow of a sphere.

The inverse curvature flow has also been studied in other ambient spaces, in particular in the hyperbolic space and in sphere, See \cite{GC4,SJ2,GC5,LW,SJ3} etc..

Flow with speed depending not only on the curvatures has recently begun to be considered. For example, flows that deform hypersurfaces by their curvature and support function were studied in \cite{IM,SWM,SJ4,GL}. In \cite{GL}, they invented a flow and proved longtime existence and smooth convergence to a round sphere when the ambient space is a space form. Meanwhile, they proved a class of Alexandrov-Fenchel inequalities of quermassintegrals. In \cite{SJ4}, they deduced a new Minkowski-type inequality in the anti-deSitter Schwarzschild manifolds and a weighted isoperimetric-type inequality in hyperbolic space.

For a certain range of $\alpha,\beta$, the limit of flows with speed $u^\alpha f^\beta$ can be an ellipsoid. For example, Andrews \cite{A9} proved that the solution will converge in $C^\infty$ to an ellipsoid along the contracting flow with the speed of $\frac{1}{n+2}$-power of the Gauss-Knonecker curvature after scaling. In \cite{IM2,IM3}, the authors studied flows of the convex hypersurfaces at the speeds of $-u^\alpha K^\beta$ and $\phi u^{2-m} K^{-1}$ respectively, where $u$ is the support function, $K$ is the Gauss curvature, $\phi$ is a smooth positive function on $\mathbb{S}$, $\alpha=\frac{(n+1)(1-p)}{n+1+p}$, $\beta=\frac{p}{n+1+p}$, $1\le p<\frac{n+1}{n-1}$, $-2\le m<\infty$ and $m\ne1$. The solutions converge to an ellipsoid.

A class of curvature flows was introduced by \cite{IM,LSW}, where the speed of the flow depends on an anisotropic factor, support function or radial function, and a curvature function. These flows can solve the $L_p$-Christoffel-Minkowski problems or dual Minkowski problems. Whether the flows can be extended is an interesting problem. In the present work \cite{DL} we also consider this kind of flow,$$\frac{\p X}{\p t}=u^\a f^{-\beta}\nu,$$
in the Euclidean space $\mathbb{R}^{n+1}$, $n\geq2$. When $f=(\frac{\sigma_n}{\sigma_k})^{\frac{1}{n-k}}$, the flow has been studied by Sheng and Yi in \cite{SWM}. 
 In our above mentioned paper \cite{DL}, we use the inverse Gauss map to re-parameterize the initial hypersurface, therefore convexity is essential. Now we improve the method for further extension.

 Let $M_0$ be a closed, smooth and star-shaped hypersurface in $\mathbb{R}^{n+1}$ ($n\geq2$), and $M_0$ encloses the origin. In the first result, we study the following expanding flow
 \begin{equation}
 	\label{1.1}
 	\begin{cases}
 		&\frac{\partial X}{\partial t}(x,t)=u^\alpha f^{-\beta}(x,t) \nu(x,t),\\
 		&X(\cdot,0)=X_0,
 	\end{cases}
 \end{equation}
 where $f(x,t)$ is a suitable curvature function of the hypersurface $M_t$ parameterized by $X(\cdot,t): M^n\times[0,T^*)\to \mR^{n+1}$, $\beta>0$, $u$ is the support function defined later and $\nu(\cdot,t)$ is the outer unit normal vector field to $M_t$.

To formulate our results, we shall suppose that the curvature function $f$ can be expressed as $f(\cdot,t)=f (\kappa_1,...,\kappa_n)$, where $\kappa_1,...,\kappa_n$ are the principal curvatures of the hypersurface $M_t$.

We obtain convergence results for a large class of speeds and therefore make the following assumption.
\begin{assumption}\label{a1.1}
	Let $\Gamma\subseteq\mathbb{R}^n$ be a symmetric, convex, open cone containing
	\begin{equation}\label{1.2}
		\Gamma_+=\{(\k_i)\in\mathbb{R}^n:\k_i>0\},
	\end{equation}
	and suppose that $f$ is positive in $\Gamma$, homogeneous of degree $1$, and concave with
	\begin{equation}\label{1.3}
		\frac{\p f}{\p\k_i}>0,\quad f\vert_{\p\Gamma=0},\quad f^{-\beta}(1,\cdots,1)=\eta.
	\end{equation}
\end{assumption}

We first prove the following
\begin{theorem}\label{t1.2}
	Assume $\alpha, \beta\in \mR$ satisfying $\alpha \le 0<\beta\le 1-\alpha$. Let $f\in C^2(\Gamma)\cap C^0(\p\Gamma)$ satisfy Assumption \ref{a1.1}, and let $X_0(M)$ be the embedding of a closed $n$-dimensional manifold $M^n$ in $\mathbb{R}^{n+1}$ such that $X_0(M)$ is a graph over $\mathbb{S}^n$, and such that $\k\in\Gamma$ for all n-tuples of principal curvatures along $X_0(M)$. Then the flow (\ref{1.1}) has a unique smooth solution $M_t$ for all time $t>0$. For each $t\in[0,\infty)$, $X(\cdot,t)$ is a parameterization of a smooth, closed, star-shaped hypersurface $M_t$ in $\mR^{n+1}$ by $X(\cdot,t)$: $M^n\to \mR^{n+1}$. After a proper rescaling $X\to \varphi^{-1}(t)X$, where
	\begin{equation}\label{1.4}
		\begin{cases}
			\varphi(t)=e^{\eta t} &\text{ if }\alpha=1-\beta,\\
			\varphi(t)=(1+(1-\beta-\alpha)\eta t)^{\frac{1}{1-\beta-\alpha}} &\text{ if }\alpha\not=1-\beta,
		\end{cases}
	\end{equation}
the hypersurface $\widetilde M_t=\varphi^{-1}M_t$ converges exponentially to a round sphere centered at the origin in the $C^\infty$-topology.
\end{theorem}

The flow (\ref{1.1}) can be described by a ODE of the support function if $\beta=0$. So we don't state that result in here.

The $k$th elementary symmetric function $\sigma_k$ is defined by
$$\sigma_k(\k_1,...,\k_n)=\sum_{1\leq i_1<\cdot\cdot\cdot<i_k\leq n}\k_{i_1}\cdot\cdot\cdot\k_{i_k},$$
and let $\sigma_0=1$.

Let us make some remarks about our conditions. The convex cone $\Gamma$ that contains the positive cone in (\ref{1.2}) is decided by $f$, e. g., $\Gamma =\{(\k_i)\in\mathbb{R}^n:\sigma_1>0\}$ if $f=\sigma_1$; $\Gamma =\{(\k_i)\in\mathbb{R}^n:\sigma_2>0 \mbox { and } \sigma_1-\k_i>0\; \forall i=1,\cdots, n\}$ if $f=\sigma_2^{\frac 12}$; $\Gamma =\Gamma_+$ the positive cone if $f=\sigma_n^{\frac 1n}$. (\ref{1.3}) ensures that this equation is parabolic.  Star-shaped initial hypersurface means it can be written as a graph over  $\mathbb{S}^n$. In particular, for $\a=0$ and $\beta=1$, this is the results of Gerhardt \cite{GC3} and Urbas \cite{UJ2}. If $\Gamma=\Gamma_+$, this is our previous work \cite{DL}.

We give some examples of functions $f$ satisfying the required hypotheses. For any integer $k,l$ such that $0\leq k< l\leq n$, $(\frac{\sigma_l}{\sigma_k})^{\frac{1}{l-k}}$ is smooth, positive, symmetric function  and homogenous of degree one on the convex cone. It is easy to check that (\ref{1.2}) and (\ref{1.3}) hold for $(\frac{\sigma_l}{\sigma_k})^{\frac{1}{l-k}}$. $(\frac{\sigma_l}{\sigma_k})^{\frac{1}{l-k}}$ satisfies the concavity by \cite{HGC}.

The second example is $f=(\sum_{i=1}^n\k_i^{k})^{\frac{1}{k}}$ for $k\ne0$. Then  $f$ is smooth, positive, symmetric functions and homogenous of degree one on the convex cone. It is easily checked that all conditions hold for $f$.

More examples can be constructed as follows:

If $f_1,\cdots,f_k$ satisfy our conditions, then $f=\prod_{i=1}^kf_i^{\alpha_i}$ also satisfies our conditions, where $\alpha_i\ge 0$ and $\sum_{i=1}^k\alpha_i=1$.
More example can be seen in \cite{B3,B4}.

The study of the asymptotic behaviour of the flow (\ref{1.1}) is equivalent to the long time behaviour of the normalised flow. Let $\widetilde X(\cdot,\tau)=\varphi^{-1}(t)X(\cdot,t)$, where
\begin{equation}\label{1.5}
	\tau=\begin{cases}
		t &\text{  if }\alpha=1-\beta,\\
		\frac{log((1-\alpha-\beta)\eta t+1)}{(1-\alpha-\beta)\eta} &\text{  if }\alpha\not=1-\beta.
	\end{cases}
\end{equation}
Then $\widetilde X(\cdot,\tau)$ satisfies the following normalized flow
\begin{equation}\label{1.6}
	\begin{cases}
		\frac{\partial \widetilde X}{\p \tau}(x,\tau)=\widetilde u^\alpha \widetilde f^{-\beta}(x,\tau)\nu-\eta \widetilde X,\\[3pt]
		\widetilde X(\cdot,0)=\widetilde X_0.
	\end{cases}
\end{equation}
For convenience we still use $t$ instead of $\tau$ to denote the time variable and omit the ``tilde'' if no confusions arise. We can find that the flow (\ref{1.6}) is equivalent (up to an isomorphism) to
\begin{equation}\label{1.7}
	\begin{cases}
		&\frac{\partial X}{\partial t}=(u^\alpha f^{-\beta}(x,t)-\eta u) \nu(x,t),\\
		&X(\cdot,0)=X_0.
	\end{cases}
\end{equation}

In order to prove Theorem \ref{t1.2}, we shall establish the a priori estimates for the normalized flow (\ref{1.7}), and show that if $X(\cdot,t)$ solves (\ref{1.7}), then the radial function $\rho$ converges exponentially to a constant as $t\to\infty$.

Secondly, we make a natural extension to the normalized flow (\ref{1.7}) to Hyperbolic space.



\begin{theorem}\label{t1.3}
Assume $\alpha, \beta\in \mR$ satisfying $\alpha \le 0<\beta< 1-\alpha$. Let $f\in C^2(\Gamma)\cap C^0(\p\Gamma)$ satisfy Assumption \ref{a1.1}, and let $X_0(M)$ be the embedding of a closed $n$-dimensional manifold $M$ in $\mathbb{H}^{n+1}$ such that $X_0(M)$ is a graph over $\mathbb{S}^n$, and such that $\k\in\Gamma$ for all n-tuples of principal curvatures along $X_0(M)$. Then any solution $X$ of (\ref{1.7}) exists for all positive times and smoothly converges exponentially to a geodesic slice in the $C^\infty$-topology.
\end{theorem}

Remark: The condition $\alpha=1+\beta$ will cause the flow (\ref{1.7}) to contract in $\mH^{n+1}$. This will be proved in Section 3. If the ambient space is sphere, the a priori estimates couldn't be established.

Next, we introduce some monotone quantities involving a weighted $\sigma_k$ integral along inverse curvature flows in the Euclidean space $\mR^{n+1}$. We denote that $$S_{\iota,k}(t)=\int_{M_t}u^\iota p_kd\mu$$ and $$T_{\iota,k}(t)=\int_{M_t}p_k^\iota d\mu,$$ where $p_k$ is defined as the normalized $k$th elementary symmetric function, i.e. $p_k=\frac{1}{C_n^k}\sigma_k$, $\iota\in\mR$ and $0\le k\le n$. It is easy to derive that $S_{1,k+1}=S_{0,k}=T_{1,k}$ is $k$th quermassintegrals by Minkowski formulas and $T_{\iota,0}=T_{0,k}=A(M)$, where $A(M)$ is the area of $M$.
\begin{theorem}\label{t1.4}
Suppose $M_t$ is a smooth solution to the inverse curvature flow $$\frac{\p X}{\p t}=\frac{p_{k-1}}{p_k}\nu-X \text{\qquad or \qquad} \frac{\p X}{\p t}=\(\frac{p_{k-1}}{p_k}-u\)\nu,$$ where $0<k\le n$. Then the following hold:

(\romannumeral1) In the case where $\iota>1$ and $M_0$ is convex, $S_{\iota,n}$ is monotone decreasing with $k=n$ and $S_{\iota,n}(t)$ is a constant function if and only if $M_t$ is a round sphere for each t.

(\romannumeral2) In the case where $\iota=1$ and $M_0$ is $k$-convex, $S_{1,k}$ is invariant for each t and $0< k\le n$;

 $S_{1,k+1}$ is monotone decreasing for $0<k<n$ and $S_{1,k+1}(t)$ is a constant function if and only if $M_t$ is a round sphere for each t;

  $S_{1,l}$ is monotone increasing for $0\le l<k\le n$ and $S_{1,l}(t)$ is a constant function if and only if $M_t$ is a round sphere for each t.

  (\romannumeral3)
  In the case where $0<\iota<1$ and $M_0$ is convex, $S_{\iota,n}$ is monotone increasing with $k=n$ and $S_{\iota,n}(t)$ is a constant function if and only if $M_t$ is a round sphere for each t.

  (\romannumeral4) In the case where $\iota=0$ and $M_0$ is $k$-convex, $S_{0,k-1}$ is invariant for each $t$ and $0< k\le n$. $S_{0,n}=\omega_n$, where $\omega_n$ is the area of the unit sphere $\mS^n$ in $\mR^{n+1}$;

  $S_{0,k}$ is monotone decreasing for $0<k<n$ and $S_{0,k}(t)$ is a constant function if and only if $M_t$ is a round sphere for each t;

  $S_{0,l}$ is monotone increasing for $0\le l<k-1\le n-1$ and $S_{0,l}(t)$ is a constant function if and only if $M_t$ is a round sphere for each t.

   (\romannumeral5) In the case where $\iota<0$ and $M_0$ is $k$-convex, $S_{\iota,k}$ is monotone decreasing for $0<k\le n$ and $S_{\iota,k}(t)$ is a constant function if and only if $M_t$ is a round sphere for each t.
\end{theorem}

\begin{theorem}\label{t1.5}
	Suppose $M_t$ is a smooth solution to the inverse curvature flow $$\frac{\p X}{\p t}=p_k^{-\frac{1}{k}}\nu-X \text{\qquad or \qquad} \frac{\p X}{\p t}=\(p_k^{-\frac{1}{k}}-u\)\nu,$$ where $0<k\le n$. Then the following hold:
	
	(\romannumeral1) In the case where $\iota>1$ and $M_0$ is $k$-convex, $T_{\iota,k}$ is monotone decreasing with $0<k\le n$ and $T_{\iota,k}(t)$ is a constant function if and only if $M_t$ is a round sphere for each t.
	
	(\romannumeral2) In the case where $\iota=1$ and $M_0$ is $k$-convex, $T_{1,k}$ is monotone decreasing with $0<k\le n$ and $T_{1,k}(t)$ is a constant function if and only if $M_t$ is a round sphere for each t;
	
	$T_{1,k-1}$ is monotone decreasing for $1<k\le n$ and $T_{1,k-1}(t)$ is a constant function if and only if $M_t$ is a round sphere for each t.
	
	(\romannumeral3)
	In the case where $0<\iota<1$ and $M_0$ is convex, $T_{\iota,n}$ is monotone increasing with $k=n$ and $S_{\iota,n}(t)$ is a constant function if and only if $M_t$ is a round sphere for each t.
	
	(\romannumeral4) In the case where $\iota=0$, $0\le l\le n$ and $M_0$ is $k$-convex, $T_{0,l}=A(M)$ is monotone increasing for $1<k\le n$ and $T_{0,l}(t)$ is a constant function if and only if $M_t$ is a round sphere for each t;
	
	$T_{0,l}$ is invariant for each t and $k=1$.
	
	
	(\romannumeral5) In the case where $\forall\iota$ and $M_0$ is $k$-convex, $T_{\iota,0}=A(M)$ is monotone increasing for $1<k\le n$ and $T_{\iota,0}(t)$ is a constant function if and only if $M_t$ is a round sphere for each t;
	
	$T_{\iota,0}$ is invariant for each t and $k=1$.
\end{theorem}

In 2009, Guan and Li \cite{GL2} used the flow $X_t=(\frac{\sigma_{k-1}}{\sigma_k}-r(t)u)\nu$ to prove the following isoperimetric inequality for quermassintegrals of non-convex starshaped domains, where $r(t)$ is a normalization constant to make $V_{n-k}(\Omega_t)$ invariant under the flow and $V_{(n+1)-k}(\Omega_t)$ is nondecreasing.
\begin{theorem}\cite{GL2}\label{t1.6}
 Suppose $\Omega\subset\mR^{n+1}$ is a smooth $k$-convex star-shaped domain. Then there holds
  \begin{equation}\label{1.8}
  	\left(\frac{V_{(n+1)-m}(\Omega)}{V_{(n+1)-m}(B)}\right)^{\frac{1}{n+1-m}}\leq\left(\frac{V_{n-m}(\Omega)}{V_{n-m}(B)}\right)^{\frac{1}{n-m}},  0\leq m\le k\le n,
  \end{equation}
where $V_{(n+1)-m}(\Omega)=\int_{\p\Omega}u\sigma_m(\k)d\mu_M$, $B$ is the unit sphere in $\mR^{n+1}$. The equality holds if and only if $\Omega$ is a ball.
\end{theorem}
It will be proved that Theorem \ref{t1.6} is a straightforward corollary by (\romannumeral2) or (\romannumeral4) of Theorem \ref{t1.4}. The locally constrained inverse curvature type flow in Theorem \ref{t1.4} was introduced in Brendle, Guan and Li in \cite{BGL}.

As natural expansions, we can derive a lot of the extensions of quermassintegral inequalities. These inequalities will be given in section 7.

At last, we list some applications about the more general flows.
\begin{theorem}\label{t1.7}
Suppose $M_t$ is a smooth solution to the inverse curvature flow $$\frac{\p X}{\p t}=\frac{p_{n-m-1}}{u^mp_n}\nu-X \text{\qquad or \qquad} \frac{\p X}{\p t}=\(\frac{p_{n-m-1}}{u^mp_n}-u\)\nu$$ for $0\le m\le n-1$. The initial hypersurface $M_0$ is convex. Then the following hold:

(\romannumeral1) For $0\le m\le k\le n-1$, $\int_{M_t}p_kd\mu$ is monotone increasing and $\int_{M_t}p_kd\mu$ is a constant function if and only if $M_t$ is a round sphere for each t.

(\romannumeral2) For $\iota\ge m+1$ or $\iota<0$, $S_{\iota,n}$ is monotone decreasing and $S_{\iota,n}$ is a constant function if and only if $M_t$ is a round sphere for each t.

(\romannumeral3) For $0<\iota\le 1$, $S_{\iota,n}$ is monotone increasing and $S_{\iota,n}$ is a constant function if and only if $M_t$ is a round sphere for each t. $S_{0,n}=\omega_n$.

According to these monotone quantities, we can get the following inequalities. For $\forall\iota\ge1$ or $\iota<0$, we have
$$\(\int_Mu^\iota p_nd\mu\)^\frac{1}{\iota}\ge\omega_n^{\frac{1}{\iota}-1}\int_M p_{n-1}d\mu\ge\omega_n^{\frac{1}{\iota}-\frac{1}{n-k}}\(\int_{M}p_kd\mu\)^\frac{1}{n-k}.$$
For $\forall r\ge1$ or $r<0$ and $0<s\le 1$, we have
$$\(\int_Mu^r p_nd\mu\)^\frac{1}{r}\ge\omega_n^{\frac{s-r}{rs}}\(\int_Mu^s p_nd\mu\)^\frac{1}{s}.$$
The equality holds if and only if $M$ is a round sphere.
\end{theorem}

The rest of the paper is organized as follows. We first recall some notations and known results in Section 2 for later use. In Section 3, we establish the a priori estimates, which ensure the long time existence of these flows. In Section 4, we show the convergence of the flow (\ref{1.7}) in $\mR^{n+1}$ and $\mH^{n+1}$, and complete the proof of Theorem \ref{t1.2} and \ref{t1.3}. In section 5, we give the proof of Theorem \ref{t1.4} and \ref{t1.5}. The proof of Theorem \ref{t1.7} is given in Section 6. Finally in Section 7, we prove \ref{t1.6} in view of the monotone quantities in Theorem \ref{t1.4}, and as application, we give a summary of inequalities involving the weighted integral of $k$th elementary symmetric function.

\section{Preliminary}
\subsection{Intrinsic curvature}
We now state some general facts about hypersurfaces, especially those that can be written as graphs. The geometric quantities of ambient spaces will be denoted by $(\bar{g}_{\alpha\beta})$, $(\bar{R}_{\alpha\beta\gamma\delta})$ etc., where Greek indices range from $0$ to $n$. Quantities for $M$ will be denoted by $(g_{ij})$, $(R_{ijkl})$ etc., where Latin indices range from $1$ to $n$. In this section, we denote the ambient spaces by $\mK^{n+1}$, which means $\mR^{n+1}$ or $\mH^{n+1}$.

Let $\nabla$, $\bar\nabla$ and $D$ be the Levi-Civita connection of $g$, $\bar g$ and the Riemannian metric $e$ of $\mathbb S^n$  respectively. All indices appearing after the semicolon indicate covariant derivatives. The $(1,3)$-type Riemannian curvature tensor is defined by
\begin{equation}\label{2.1}
	R(U,Y)Z=\nabla_U\nabla_YZ-\nabla_Y\nabla_UZ-\nabla_{[U,Y]}Z,
\end{equation}
or with respect to a local frame $(e_i)$,
\begin{equation}\label{2.2}
	R(e_i,e_j)e_k={R_{ijk}}^{l}e_l,
\end{equation}
where we use the summation convention (and will henceforth do so). The coordinate expression of (\ref{2.1}), the so-called Ricci identities, read
\begin{equation}\label{2.3}
	Y_{;ij}^k-Y_{;ji}^k=-{R_{ijm}}^kY^m
\end{equation}
for all vector fields $Y=(Y^k)$. We also denote the $(0,4)$ version of the curvature tensor by $R$,
\begin{equation}\label{2.4}
	R(W,U,Y,Z)=g(R(W,U)Y,Z).
\end{equation}
\subsection{Extrinsic curvature}
The induced geometry of $M$ is governed by the following relations. The second fundamental form $h=(h_{ij})$ is given by the Gaussian formula
\begin{equation}\label{2.5}
	\bar\nabla_ZY=\nabla_ZY-h(Z,Y)\nu,
\end{equation}
where $\nu$ is a local outer unit normal field. Note that here (and in the rest of the paper) we will abuse notation by disregarding the necessity to distinguish between a vector $Y\in T_pM$ and its push-forward $X_*Y\in T_p\mathbb{K}^{n+1}$. The Weingarten endomorphism $A=(h_j^i)$ is given by $h_j^i=g^{ki}h_{kj}$, and the Weingarten equation
\begin{equation}\label{2.6}
	\bar\nabla_Y\nu=A(Y),
\end{equation}
holds there, or in coordinates
\begin{equation}\label{2.7}
	\nu_{;i}^\alpha=h_i^kX_{;k}^\alpha.
\end{equation}
We also have the Codazzi equation in $\mathbb{K}^{n+1}$
\begin{equation}\label{2.8}
	\nabla_Wh(Y,Z)-\nabla_Zh(Y,W)=-\bar{R}(\nu,Y,Z,W)=0
\end{equation}
or
\begin{equation}\label{2.9}
	h_{ij;k}-h_{ik;j}=-\bar R_{\alpha\beta\gamma\delta}\nu^\alpha X_{;i}^\beta X_{;j}^\gamma X_{;k}^\delta=0,
\end{equation}
and the Gauss equation
\begin{equation}\label{2.10}
	R(W,U,Y,Z)=\bar{R}(W,U,Y,Z)+h(W,Z)h(U,Y)-h(W,Y)h(U,Z)
\end{equation}
or
\begin{equation}\label{2.11}
	R_{ijkl}=\bar{R}_{\alpha\beta\gamma\delta}X_{;i}^\alpha X_{;j}^\beta X_{;k}^\gamma X_{;l}^\delta+h_{il}h_{jk}-h_{ik}h_{jl},
\end{equation}
where
\begin{equation}\label{2.12}
	\bar{R}_{\alpha\beta\gamma\delta}=-K(\bar{g}_{\alpha\gamma}\bar{g}_{\beta\delta}-\bar{g}_{\alpha\delta}\bar{g}_{\beta\gamma}),
\end{equation}
and
\begin{equation*}
	K=\begin{cases}
		-1 &\text{   in  } \mH^{n+1},\\
		0 &\text{   in  } \mR^{n+1}.
	\end{cases}
\end{equation*}
\subsection{Hypersurface in $\mathbb{K}^{n+1}$}
It is known that the space form can be viewed as Euclidean space $\mathbb{R}^{n+1}$ equipped with a metric tensor, i.e., $\mathbb{K}^{n+1}=(\mathbb{R}^{n+1},ds^2)$ with proper choice $ds^2$. More specifically, let $\mathbb S^n$ be the unit sphere in Euclidean space $\mathbb{R}^{n+1}$ with standard induced metric $dz^2$, then
\begin{equation*}
	\bar{g}:=ds^2=d\rho^2+\phi^2(\rho)dz^2,
	\end{equation*}
where
\begin{equation*}
\phi(\rho)=
\begin{cases}
sinh(\rho)&\text{  in  } \mH^{n+1},\\
\rho &\text{   in  } \mR^{n+1},
\end{cases}
\end{equation*}
 $\rho\in[0,\infty)$. Consider the vector field $V=\phi(\rho)\frac{\p}{\p\rho}$ on $\mathbb{K}^{n+1}$. We know that $V$ is a conformal killing field. By \cite{GL}, we have the following lemma.
\begin{lemma}\label{l2.1}
	The vector field $V$ satisfies $\bar{\nabla}_XV=\phi'(\rho)X$.
\end{lemma}
We call the inner product $u:=<V,\nu>$ to be the support function of a hypersurface in $\mathbb{K}^{n+1}$, where $<\cdot,\cdot>=\bar{g}(\cdot,\cdot)$. Then we can derive the gradient and hessian of the support function $u$ under the induced metric $g$ on $M$.
\begin{lemma}\label{l2.2}
	The support function $u$ satisfies
	\begin{equation}\label{2.13}
		\begin{split}
		\nabla_iu=&g^{kl}h_{ik}\nabla_l\Phi,  \\
		 \nabla_i\nabla_ju=&g^{kl}\nabla_kh_{ij}\nabla_l\Phi+\phi'h_{ij}-(h^2)_{ij}u,
		 \end{split}
	\end{equation}
where $(h^2)_{ij}=g^{kl}h_{ik}h_{jl},$ and
\begin{equation*}
	\Phi(\rho)=\int_{0}^{\rho}\phi(r)dr=
	\begin{cases}
		cosh(\rho)&\text{  in  } \mH^{n+1},\\
		\frac{1}{2}\rho^2 &\text{   in  } \mR^{n+1}.
	\end{cases}
\end{equation*}
\end{lemma}
The proof of Lemma \ref{l2.2} can be seen in \cite{GL,BLO,JL}.
\subsection{Graphs in $\mathbb{K}^{n+1}$}
Let $(M,g)$ be a hypersurface in $\mathbb{K}^{n+1}$ with induced metric $g$. We now give the local expressions of the induced metric, second fundamental form, Weingarten curvatures etc when $M$ is a graph of a smooth and positive function $\rho(z)$ on $\mathbb{S}^n$. Let $\p_1,\cdots,\p_n$ be a local frame along $M$ and $\p_\rho$ be the vector field along radial direction. Then the support function, induced metric, inverse metric matrix, second fundamental form can be expressed as follows (\cite{GL}).
\begin{align*}
	u &= \frac{\phi^2}{\sqrt{\phi^2+|D\rho|^2}},\;\; \nu=\frac{1}{\sqrt{1+\phi^{-2}|D\rho|^2}}(\frac{\p}{\p\rho}-\phi^{-2}\rho_i\frac{\p}{\p x_i}),   \\
	g_{ij} &= \phi^2e_{ij}+\rho_i\rho_j,  \;\;   g^{ij}=\frac{1}{\phi^2}(e^{ij}-\frac{\rho^i\rho^j}{\phi^2+|D\rho|^2}),\\
	h_{ij} &=\(\sqrt{\phi^2+|D\rho|^2}\)^{-1}(-\phi D_iD_j\rho+2\phi'\rho_i\rho_j+\phi^2\phi'e_{ij}),\\
	h^i_j &=\frac{1}{\phi^2\sqrt{\phi^2+|D\rho|^2}}(e^{ik}-\frac{\rho^i\rho^k}{\phi^2+|D\rho|^2})(-\phi D_kD_j\rho+2\phi'\rho_k\rho_j+\phi^2\phi'e_{kj}),
	\end{align*}
where $e_{ij}$ is the standard spherical metric. It will be convenient if we introduce a new variable $\gamma$ satisfying $$\frac{d\gamma}{d\rho}=\frac{1}{\phi(\rho)}.$$
Let $\omega:=\sqrt{1+|D\gamma|^2}$, one can compute the unit outward normal $$\nu=\frac{1}{\omega}(1,-\frac{\gamma_1}{\phi},\cdots,-\frac{\gamma_n}{\phi})$$ and the general support function $u=<V,\nu>=\frac{\phi}{\omega}$. Moreover,
\begin{align}\label{2.14}
	g_{ij} &=\phi^2(e_{ij}+\gamma_i\gamma_j), \;\; g^{ij}=\frac{1}{\phi^2}(e^{ij}-\frac{\gamma^i\gamma^j}{\omega^2}),\notag\\
	h_{ij} &=\frac{\phi}{\omega}(-\gamma_{ij}+\phi'\gamma_i\gamma_j+\phi'e_{ij}),\notag\\
	h^i_j &=\frac{1}{\phi\omega}(e^{ik}-\frac{\gamma^i\gamma^k}{\omega^2})(-\gamma_{kj}+\phi'\gamma_k\gamma_j+\phi'e_{kj})\notag\\
	&=\frac{1}{\phi\omega}(\phi'\delta^i_j-(e^{ik}-\frac{\gamma^i\gamma^k}{\omega^2})\gamma_{kj}).
	\end{align}
Covariant differentiation with respect to the spherical metric is denoted by indices.

There is also a relation between the second fundamental form and the radial function on the hypersurface. Let $\widetilde{h}=\phi'\phi e$. Then
\begin{equation}\label{2.15}
	\omega^{-1}h=-\nn^2\rho+\widetilde{h}
\end{equation}
holds; cf. \cite{GC2}. Since the induced metric is given by
$$g_{ij}=\phi^2 e_{ij}+\rho_i\rho_j, $$
we obtain
\begin{equation}\label{2.16}
	\omega^{-1}h_{ij}=-\rho_{;ij}+\frac{\phi'}{\phi}g_{ij}-\frac{\phi'}{\phi}\rho_{i}\rho_{j}.
\end{equation}

We now consider the flow equation (\ref{1.7}) of radial graphs over $\mathbb{S}^n$ in $\mathbb{K}^{n+1}$. It is known (\cite{GC2}) if a closed hypersurface which is a radial graph and satisfies
$$\p_tX=\mathscr{F}\nu,$$
then the evolution of the scalar function $\rho=\rho(X(z,t),t)$ satisfies
$$\p_t\rho=\mathscr{F}\omega.$$
Thus we only need to consider the following parabolic initial value problem on $\mathbb{S}^n$,
\begin{equation}\label{2.17}
     \begin{cases}
     	\p_t\rho&=(u^\alpha f^{-\beta}-\eta u)\omega, \;\;(z,t)\in\mathbb{S}^n\times [0,\infty),\\
     	 \rho(\cdot,0)&=\rho_0,
     	\end{cases}
\end{equation}
   where $\rho_0$ is the radial function of the initial hypersurface.

   Equivalently, the equation for $\gamma$ satisfies
   \begin{equation}\label{2.18}
   	\p_t\gamma=(u^\alpha f^{-\beta}-\eta u)\frac{\omega}{\phi}.
   \end{equation}
Lastly, we can derive a connection between $\vert\nn\rho\vert$ and $\vert D\gamma\vert$.
\begin{lemma}\label{l2.3}
If $M$ is a star-shaped hypersurface, we can derive that $\vert\nn\rho\vert^2=1-\frac{1}{\omega^2}$.
\end{lemma}
\begin{proof}
We can derive this lemma directly via calculations.
\begin{equation}\label{2.19}
\begin{split}
\vert\nn\rho\vert^2&=g^{ij}\rho_{;i}\rho_{;j}\\
&=\frac{1}{\phi^2}(e^{ij}-\frac{\rho_i\rho_j}{\phi^2\omega^2})\rho_i\rho_j\\
&=\phi^{-2}(\vert D\rho\vert^2-\frac{\vert D\rho\vert^4}{\phi^2\omega^2})\\
&=\vert D\gamma\vert^2(1-\frac{\vert D\gamma\vert^2}{\omega^2})\\
&=1-\frac{1}{\omega^2}.
\end{split}
\end{equation}
	\end{proof}
\subsection{Elementary symmetric functions}
We review some properties of elementary symmetric functions. See \cite{HGC} for more details.

In Section 1 we give the definition of elementary symmetric functions. The definition can be extended to symmetric matrices. Let $A\in Sym(n)$ be an $n\times n$ symmetric matrix. Denote by $\k=\k(A)$ the eigenvalues of $A$. Set $p_m(A)=p_m(\k(A))$. We have
\begin{equation*}
p_m(A)=\frac{(n-m)!}{n!}\delta^{j_1\cdots j_m}_{i_1\cdots i_m}A_{i_1j_1}\cdots A_{i_mj_m}, \qquad m=1,\cdots,n.
\end{equation*}
\begin{lemma}\label{l2.4}
Denote $p_m^{ij}=\frac{\p p_m}{\p A_{ij}}$. Then we have
\begin{align}
p_m^{ij}A_{ij}&=mp_m,\\
p_m^{ij}\delta_j^i&=mp_{m-1},\\
p_m^{ij}(A^2)_{ij}&=np_1p_m-(n-m)p_{m+1},
\end{align}
where $(A^2)_{ij}=\sum_{k=1}^nA_{ik}A_{kj}$.
\end{lemma}
\begin{lemma}\label{l2.5}
If $\k\in\Gamma_m^+=\{x\in\mR^n:p_i(x)>0, i=1,\cdots,m\}$, we have the following Newton-MacLaurin inequalities.
\begin{align}
p_{m+1}(\k)p_{k-1}(\k)&\le p_k(\k)p_m(\k),\\
p_1\ge p_2^{\frac{1}{2}}\ge\cdots&\ge p_m^{\frac{1}{m}},\qquad 1\le k\le m.
\end{align}
Equality holds if and only if $\k_1=\cdots=\k_n$.
\end{lemma}
Let us denote by $\sigma_{k,i}(\k)$ the sum of the terms of $\sigma_k(\k)$ not containing the factor $\k_i$. Then the following identities hold.
\begin{proposition}\label{p2.6}
\cite{HGC} We have, for any $k=0,\cdots,n$, $i=1,\cdots,n$ and $\k\in\mR^n$,
\begin{align}
\frac{\p\sigma_{k+1}}{\p\k_i}(\k)&=\sigma_{k,i}(\k),\\
\sigma_{k+1}(\k)&=\sigma_{k+1,i}(\k)+\k_i\sigma_{k,i}(\k),\\
\sum_{i=1}^n\sigma_{k,i}(\k)&=(n-k)\sigma_k(\k),\\
\sum_{i=1}^n\k_i\sigma_{k,i}(\k)&=(k+1)\sigma_{k+1}(\k),\\
\sum_{i=1}^n\k_i^2\sigma_{k,i}(\k)&=\sigma_1(\k)\sigma_{k+1}(\k)-(k+2)\sigma_{k+2}(\k).
\end{align}
\end{proposition}

\section{A Priori Estimates}
In this section, we establish the priori estimates and show that the flow exists for long time.
For convenience, we denote that $\Psi=u^\a,G=f^{-\beta}$, then the equation (\ref{1.1}) can be written in the following form
$$\frac{\p X}{\p t}=u^\alpha f^{-\beta}(x,t) \nu(x,t)=\Psi G\nu.$$

We first show the $C^0$-estiamte of the solution to (\ref{2.17}).

\begin{lemma}\label{l3.1}
Let $\rho(x,t)$, $t\in[0,T)$, be a smooth, star-shaped solution to (\ref{2.17}). If (\romannumeral1) $0<\beta<1-\alpha$ in $\mH^{n+1}$ or (\romannumeral2) $0<\beta\le1-\alpha$ in $\mR^{n+1}$, then there is a positive constant $C_1$ depending only on $\alpha,\beta$ and the lower and upper bounds of $\rho(\cdot,0)$ such that
\begin{equation*}
\frac{1}{C_1}\leq \rho(\cdot,t)\leq C_1.
\end{equation*}
\end{lemma}
\begin{proof}
Let $\rho_{\max}(t)=\max_{z\in \mS^n}\rho(\cdot,t)=\rho(z_t,t)$. For fixed time $t$, at the point $z_t$, we have $$D_i\rho=0 \text{ and } D^2_{ij}\rho\leq0.$$
Note that $\omega=1$, $u=\frac{\phi}{\omega}=\phi$ and
\begin{equation}\label{3.1}
	\begin{split}
h^i_j &=\frac{1}{\phi^2\sqrt{\phi^2+|D\rho|^2}}(e^{ik}-\frac{\rho^i\rho^k}{\phi^2+|D\rho|^2})(-\phi D_kD_j\rho+2\phi'\rho_k\rho_j+\phi^2\phi'e_{kj})\\
&=-\phi^{-2}\rho_{ij}+\phi^{-1}\phi'e_{ij}.
\end{split}
\end{equation}
 (\romannumeral1) In $\mH^{n+1}$ case, at the point $z_t$,
 we have $F^{-\beta}(h^i_j)\leq\eta(\frac{\phi'}{\phi})^{-\beta}\leq\eta$ by $\frac{\phi'}{\phi}\geq1$. We denote
  \begin{equation}\label{3.2}
 	F([a_{ij}])=f(\mu_1,\cdots,\mu_n),
 \end{equation}
 where $\mu_1,\cdots,\mu_n$ are the eigenvalues of matrix $[a_{ij}]$. It is not difficult to see that the eigenvalues of $[F^{ij}]=[\frac{\p F}{\p a_{ij}}]$ are $\frac{\p f}{\p\mu_1},\cdots,\frac{\p f}{\p\mu_n}$.
 Thus
$$\frac{d}{dt}\rho_{\max}\leq\eta\phi(\phi^{\alpha-1}-1),$$
where $\phi=sinh\rho$. Hence there exists a constant $C_0$ such that $\rho_{\max}\leq \max\{C_0,\rho_{\max}(0)\}$.

Similarly, let $\rho_{\min}(t)=\min_{z\in \mS^n}\rho(\cdot,t)=\rho(z_t,t)$. By (\ref{3.1}), we have $F^{-\beta}(h^i_j)\geq\eta(\frac{\phi'}{\phi})^{-\beta}$. Then
\begin{equation*}
	\frac{d}{dt}\rho_{\min}\geq\eta\phi(C_2\phi^{\alpha+\beta-1}-1),
	\end{equation*}
where $C_2\le\phi'^{-\beta}(\rho_{\max})<1$. By $\alpha+\beta<1$ we have $\rho_{\min}\geq \min\{\frac{1}{C_0},\rho_{\min}(0)\}$.

 (\romannumeral2) In $\mR^{n+1}$ case, at the point $z_t$, we have $F^{-\beta}(h^i_j)\leq\eta(\frac{\phi'}{\phi})^{-\beta}$. Then
 $$\frac{d}{dt}\rho_{\max}\leq\eta\rho(\rho^{\alpha+\beta-1}-1),$$
 hence $\rho_{\max}\leq \max\{C_0,\rho_{\max}(0)\}$.
 Similarly, $$\frac{d}{dt}\rho_{\min}\geq\eta\rho(\rho^{\alpha+\beta-1}-1).$$
 By $\alpha+\beta\le1$ we have $\rho_{\min}\geq \min\{\frac{1}{C_0},\rho_{\min}(0)\}$.
\end{proof}
Remark: If $\alpha=1-\beta$ and the ambient space is the hyperbolic space $\mH^{n+1}$, we look at the flow (\ref{1.7}) with initial hypersurface $X(0)=S_{\rho_0}=\{X^0=\rho_0\}$. The geodesic spheres are totally umbilical and their second fundamental form is given by ${h^i_j}=coth\rho\delta^i_j$. Then the flow hypersurfaces $M(t)$ will be spheres with radii $\rho(t)$ satisfying the scalar curvature flow equation
$$\p_t\rho=\eta\phi(cosh^{-\beta}\rho-1)\leq0,$$
equality holds if and only if $\rho=0$. At this situation, the flow will contract and exist for finite time.

If $\a+\beta-1\le0$ and the ambient space is $\mS^{n+1}$, we also consider (\ref{1.7}) with $X(0)=S_{\rho_0}=\{X^0=\rho_0\}$.
We have ${h^i_j}=cot\rho\delta^i_j$. Then
$$\p_t\rho=\eta sin\rho(sin^{\a+\beta-1}\rho cos^{-\beta}\rho-1)>0.$$
 At this situation, the flow will expand to infinity. Similarly, if $\a+\beta-1>0$, the flow also don't have a good convergence.

Let $M(t)$ be a smooth family of closed hypersurfaces in $\mK^{n+1}$. Let $X(\cdot,t)$ denote a point on $M(t)$. In general, we have the following evolution property.
\begin{lemma}\label{l3.2}
	Let $M(t)$ be a smooth family of closed hypersurfaces in $\mK^{n+1}$ evolving along the flow$$\p_tX=\mathscr{F}\nu,$$
	where $\nu$ is the unit outward normal vector field and $\mathscr{F}$ is a function defined on $M(t)$. Then we have the following evolution equations.
	\begin{equation}\label{3.3}
		\begin{split}
			\p_tg_{ij}&=2\mathscr{F}h_{ij},\\
			\p_t\nu&=-\nn\mathscr{F},\\
			\p_td\mu_g&=\mathscr{F}Hd\mu_g,\\
			\p_th_{ij}&=-\nn_i\nn_j\mathscr{F}+\mathscr{F}(h^2)_{ij}-K\mathscr{F}g_{ij},\\
			\p_tu&=\phi'\mathscr{F}-<\nn \Phi,\nn\mathscr{F}>,
			\end{split}
	\end{equation}
where $d\mu_g$ is the volume element of the metric $g(t)$, $(h^2)_{ij}=h_i^kh_{kj}$.
\end{lemma}
\begin{proof}
	Proof is standard, see for example, \cite{HG}.
\end{proof}

\begin{lemma}\label{l3.3}
 Let $0<\beta\leq1-\alpha$, and $X(\cdot,t)$ be the solution to the flow (\ref{1.7}) which encloses the origin for $t\in[0,T)$. Then there is a positive constant $C_3$ depending on the initial hypersurface and $\alpha,\beta$,  such that
$$\frac{1}{C_3}\leq u^{\alpha-1}F^{-\beta}\leq C_3.$$
\end{lemma}
\begin{proof}
Consider the auxiliary function
$$Q=u^{\alpha-1}F^{-\beta}.$$
Then $Q=u^{\alpha-1}G$ and $G$ is homogenous of degree $-\beta$.
$$(u^{\alpha}G)_{ij}=(Qu)_{ij}=Q_{ij}u+Q_iu_j+Q_ju_i+Qu_{ij}.$$
In order to calculate the evolution equation of $Q$, we need to deduce the evolution equations of $u$ and $F$ first. We denote that $\bar{g}(\cdot,\cdot)=<\cdot,\cdot>$.

	\begin{align}\label{3.4}
	\frac{\p u}{\p t}&=\frac{\p}{\p t}<V,\nu>\notag\\
	&=<\bar{\nabla}_{\frac{\p x}{\p t}}V,\nu>+<V,\frac{\p \nu}{\p t}>\notag\\
	&=(Qu-\eta u)\phi'-<V,\nn(Qu-\eta u)>\notag\\
	&=u\phi'(Q-\eta)+\eta<V,\nn u>-Q<V,\nn u>-u<V,\nn Q>\notag\\
	&=u\phi'(Q-\eta)+(\eta-Q)<V,\nn u>-u<V,\nn Q>.
   	\end{align}

\begin{equation}\label{3.5}
	\begin{split}
		\frac{\p F}{\p t}=&F^{ij}\frac{\p h^i_j}{\p t}\\
		=&F^{ij}\(-\nn^i\nn_j(Qu-\eta u)+(Qu-\eta u)(h^2)^i_j-K(Qu-\eta u)\delta^i_{j}-2(Qu-\eta u)(h^2)^i_j\)\\
		=&F^{ij}\(-(u\nn^i\nn_jQ+2\nn^iQ\nn_ju+Q\nn^i\nn_ju-\eta\nn^i\nn_ju)-K(Qu-\eta u)\delta^i_{j}\\
		&-(Qu-\eta u)(h^2)^i_j\)\\
		=&F^{ij}\((\eta-Q)\nn^i\nn_ju-2\nn^iQ\nn_ju-u\nn^i\nn_jQ-u(Q-\eta)(K\delta^i_{j}+(h^2)^i_j)\).
			\end{split}
\end{equation}
At the point where $Q$ attains its spatial maximum or minimum, $\nn Q=0$. We use (\ref{2.13}), (\ref{3.4}) and (\ref{3.5}) to deduce
\begin{equation*}
\begin{split}
\p_tQ=&\p_t(u^{\alpha-1}G)\\
=&(\alpha-1)u^{\alpha-2}\frac{\p u}{\p t}G-\beta u^{\alpha-1}F^{-\beta-1}\frac{\p F}{\p t}\\
=&(\alpha-1)u^{\alpha-2}G(u\phi'(Q-\eta)+(\eta-Q)<V,\nn u>)\\
&-\beta u^{\alpha-1}\frac{G}{F}F^{ij}\((\eta-Q)\nn^i\nn_ju-u\nn^i\nn_jQ-u(Q-\eta)(K\delta^i_{j}+(h^2)^i_j)\)\\
=&(\alpha-1)\phi'Q(Q-\eta)+(\alpha-1)\frac{Q}{u}(\eta-Q)<V,\nn u>\\
&-\beta\frac{Q}{F}\(\phi'F+<V,\nn F>-uF^{ij}(h^2)^i_j\)(\eta-Q)+\beta\frac{uQ}{F}F^{ij}\nn^i\nn_jQ\\
&+\beta\frac{uQ}{F}F^{ij}(Q-\eta)(K\delta^i_{j}+(h^2)^i_j).
\end{split}
\end{equation*}
By $$<V,\nn Q>=(\alpha-1)\frac{Q}{u}<V,\nn u>-\beta\frac{Q}{F}<V,\nn F>,$$
we have
\begin{equation*}
	\begin{split}
\frac{\p Q}{\p t}=&\beta\frac{uQ}{F}F^{ij}\nn^i\nn_jQ+(\alpha+\beta-1)\phi'Q(Q-\eta)+K\beta\frac{uQ}{F}(Q-\eta)\sum_{i=1}^{n}F^{ii}\\
=&\beta\frac{uQ}{F}F^{ij}\nn^i\nn_jQ-Q(Q-\eta)\((1-\alpha-\beta)\phi'-K\beta\sum_{i=1}^{n}F^{ii}\frac{u}{F}\).		
\end{split}
	\end{equation*}

If $1-\alpha-\beta=0$ and $K=0$, the proof is completed by the strong maximum principle. If $K\ne0$ or $1-\alpha-\beta>0$, we get $(1-\alpha-\beta)\phi'-K\beta\sum_{i=1}^{n}F^{ii}\frac{u}{F}>0$. It means that the sign of the coefficient of the highest order term $Q^2$ is negative and the sign of the coefficient of the lower order term $Q$ is positive. Applying the maximum principle we know that $\frac{1}{C_3}\leq Q\leq C_3$, where $C_3$ is a positive constant depending on the initial hypersurface and $\alpha,\beta$.
\end{proof}

According to Lemma \ref{l3.3} and $u=\frac{\phi}{\omega}\leq\phi_{\max}\leq C_4$, we can get $f\geq \frac{1}{C_5}$. That is,
\begin{corollary}\label{c3.4}
Let $0<\beta\leq1-\alpha$, and $X(\cdot,t)$ be the solution to the flow (\ref{1.7}) which encloses the origin for $t\in[0,T)$. Then there is a positive constant $C_5$ depending on the initial hypersurface and $\alpha,\beta$,  such that
$$f\geq \frac{1}{C_5}.$$
	\end{corollary}
We would like to get the upper bound of $|D\gamma|$ to show that $u$ has a lower bound.

\begin{lemma}\label{l3.5}
Let $0<\beta\leq1-\alpha$, and $X(\cdot,t)$ be the solution to the flow (\ref{1.7}) which encloses the origin for $t\in[0,T)$. Then there is a positive constant $C_6$ depending on the initial hypersurface and $\alpha,\beta$,  such that
$$\vert D \g\vert\leq C_6.$$
\end{lemma}
\begin{proof}
Consider the auxiliary function $O=\frac{1}{2}\vert D\g\vert^2$. At the point where $O$ attains its spatial maximum, we have
\begin{gather*}
D\omega=0,\\
0=D_iO=\sum_{l}\g_{li}\g_{l},\\
0\geq D_{ij}^2O=\sum_{l}\g_{li}\g_{lj}+\sum_{l}\g_{l}\g_{lij}.
\end{gather*}
By (\ref{2.14}) and (\ref{2.18}), we deduce
\begin{equation}\label{3.6}
	\begin{split}
\p_t\gamma=&(u^\alpha F^{-\beta}-\eta u)\frac{\omega}{\phi}\\
=&\frac{\phi^{\a-1}}{\omega^{\a-1}}F^{-\beta}\(\frac{1}{\phi\omega}(\phi'\delta_{ij}-(e^{ik}-\frac{\gamma_i\gamma_k}{\omega^2})\gamma_{kj})\)-\eta\\
=&\frac{\phi^{\a+\beta-1}}{\omega^{\a-\beta-1}}G-\eta.
	\end{split}
\end{equation}
We remark that here $G=F^{-\beta}([\phi'\delta_{ij}-(e^{ik}-\frac{\gamma_i\gamma_k}{\omega^2})\gamma_{kj}])$ and $$G^{ij}=G^{ij}([\phi'\delta_{ij}-(e^{ik}-\frac{\gamma_i\gamma_k}{\omega^2})\gamma_{kj}]).$$
Due to (\ref{3.6}), we have
\begin{align}\label{3.7}
\p_tO_{\max}=&\sum\g_l\g_{tl}\notag\\
=&\g_l\(\frac{(\a+\beta-1)\phi^{\a+\beta-1}\phi'\g_l}{\omega^{\a-\beta-1}}G\notag\\
&+\frac{\phi^{\a+\beta-1}}{\omega^{\a-\beta-1}}G^{ij}(\phi''\phi\g_l\delta_{ij}-\g_{kjl}(e^{ik}-\frac{\gamma_i\gamma_k}{\omega^2})+\g_{kj}(\frac{\gamma_{il}\gamma_k}{\omega^2}+\frac{\gamma_i\gamma_{kl}}{\omega^2}))\)\notag\\
=&\frac{\phi^{\a+\beta-1}}{\omega^{\a-\beta-1}}\((\a+\beta-1)\phi'\vert D\g\vert^2G+\phi''\phi\vert D\g\vert^2\sum_{i=1}^{n}G^{ii}-G^{ij}\g_l\g_{kjl}(e^{ik}-\frac{\gamma_i\gamma_k}{\omega^2})\notag\\
&+G^{ij}\g_l\g_{kj}(\frac{\gamma_{il}\gamma_k}{\omega^2}+\frac{\gamma_i\gamma_{kl}}{\omega^2})\)\notag\\
=&\frac{\phi^{\a+\beta-1}}{\omega^{\a-\beta-1}}\((\a+\beta-1)\phi'\vert D\g\vert^2G+\phi''\phi\vert D\g\vert^2\sum_{i=1}^{n}G^{ii}\notag\\
&-G^{ij}\g_l\g_{kjl}(e^{ik}-\frac{\gamma_i\gamma_k}{\omega^2})\),
\end{align}
where we have used $\sum_{l}\g_{li}\g_{l}=0$ in the last step. By the Ricci identity,$$D_l\g_{ij}=D_j\g_{li}+\e_{il}\g_j-\e_{ij}\g_l,$$
we get
\begin{align}\label{3.8}
-G^{ij}\g_l\g_{kjl}(e^{ik}-\frac{\gamma_i\gamma_k}{\omega^2})=&-G^{ij}\g_l(\g_{lkj}+\g_j\e_{lk}-\g_l\e_{kj})(e^{ik}-\frac{\gamma_i\gamma_k}{\omega^2})\notag\\
\leq&-G^{ij}(-\g_{lk}\g_{lj}+\g_l\g_je_{lk}-\vert D\g\vert^2e_{kj})(e^{ik}-\frac{\gamma_i\gamma_k}{\omega^2})\notag\\
\leq&-G^{ij}(-\g_{li}\g_{lj}+\g_i\g_j-\vert D\g\vert^2\delta_{ij}).
\end{align}
Note that $[G^{ij}]$ is negative definite. According to $\phi''\phi+1=\phi'^2$, (\ref{3.7}) and (\ref{3.8}), we can derive
\begin{align}\label{3.9}
\p_tO_{\max}\leq&\frac{\phi^{\a+\beta-1}}{\omega^{\a-\beta-1}}\((\a+\beta-1)\phi'\vert D\g\vert^2G+\phi''\phi\vert D\g\vert^2\sum_{i=1}^{n}G^{ii}+G^{ij}\g_{li}\g_{lj}\notag\\
&-G^{ij}\g_i\g_j+\vert D\g\vert^2\sum_{i=1}^{n}G^{ii}\)\notag\\
\leq&\frac{\phi^{\a+\beta-1}}{\omega^{\a-\beta-1}}\((\a+\beta-1)\phi'\vert D\g\vert^2G+\phi'^2\vert D\g\vert^2\sum_{i=1}^{n}G^{ii}+G^{ij}\g_{li}\g_{lj}-G^{ij}\g_i\g_j\).
\end{align}
In terms of the negative definite of the symmetric matrix $[G^{ij}]$, and $\a+\beta-1\leq0$, we have $G^{ij}\g_{li}\g_{lj}\leq0$ and $-G^{ij}\g_i\g_j\leq-\min_iG^{ii}\vert D\g\vert^2$. Thus
$$\p_tO_{\max}\leq0.$$
Then we have $\vert D\g(\cdot,t)\vert\leq C_6$ for a positive constant $C_6$.
\end{proof}
By Lemma \ref{l3.5}, we can get the bound of $u$ and $f$.
\begin{corollary}\label{c3.6}
Let $0<\beta\leq1-\alpha$, and $X(\cdot,t)$ be the solution to the flow (\ref{1.7}) which encloses the origin for $t\in[0,T)$. Then there are positive constants $C_4$ and $C_5$ depending on the initial hypersurface and $\alpha,\beta$,  such that
$$\frac{1}{C_4}\leq u\leq C_4\quad\text{and}\quad\frac{1}{C_5}\leq f\leq C_5.$$
\end{corollary}
\begin{proof}
By Lemma \ref{l3.5} and $u=\frac{\phi}{\omega}\leq\phi_{\max}\leq C_4$, we can derive the bound of $u$. By Lemma \ref{l3.3} and the bounds of $u$, we can get the bounds of $f$.
\end{proof}
The next step in our proof is the derivation of the principal curvature boundary.

\begin{lemma}\label{l3.7}
Let $\alpha\leq0<\beta\leq1-\alpha$, and $X(\cdot,t)$ be a smooth, closed and star-shaped solution to the flow (\ref{1.7}) which encloses the origin for $t\in[0,T)$. Then there is a positive constant $C_7$ depending on the initial hypersurface and $\alpha,\beta$,  such that the principal curvatures of $X(\cdot,t)$ are uniformly bounded from above $$\k_i(\cdot,t)\leq C_7 \text{ \qquad } \forall1\le i\le n,$$
and hence, are compactly contained in $\Gamma$, in view of Corollary \ref{c3.6}.
\end{lemma}
\begin{proof}




First, we shall prove that $\k_i$  is bounded from above by a positive constant. The principal curvatures of $M_t$ are the eigenvalues of $\{h_{il}g^{lj}\}$.

Define the functions
\begin{gather}\label{3.10}
	W(x,t)=\max\{h_{ij}(x,t)\zeta^i\zeta^j: g_{ij}(x)\zeta^i\zeta^j=1\},\\
	p(u)=-log(u-\frac{1}{2}\min u),
\end{gather}
and
\begin{equation}\label{3.12}
	\theta=log W+p(u)+N\rho,
\end{equation}
where $N$ will be chosen later. Note that
\begin{equation}\label{3.13}
	1+p'u=\frac{-\frac{1}{2}\min u}{u-\frac{1}{2}\min u}<0.
\end{equation}
We wish to bound $\theta$ from above. Thus, suppose $\theta$ attains a maximal value at $(\xi_0,t_0)\in M\times(0,T_0]$, $T_0<T^*$. Choose Riemannian normal coordinates in $(\xi_0,t_0)$, such that in this point we have
\begin{equation}\label{3.14}
	g_{ij}=\delta_{ij},\quad h^i_j=h_{ij}=\k_i\delta_{ij},\quad\k_1\geq\cdots\geq\k_n.
\end{equation}
Since $W$ is only continuous in general, we need to find a differentiable version instead. Set
$$\widetilde{W}=\frac{h_{ij}\tilde{\zeta}^i\tilde{\zeta}^j}{g_{ij}\tilde{\zeta}^i\tilde{\zeta}^j},$$
where $\tilde{\zeta}=(\tilde{\zeta}^i)=(1,0,\cdots,0).$

At $(\xi_0,t_0)$ we have
\begin{equation}\label{3.15}
	h_{11}=h^1_1=\k_1=W=\widetilde{W}
\end{equation}
and in a neighborhood of $(\xi_0,t_0)$ there holds
$$\widetilde{W}\leq W.$$
Using $h^1_1=h_{1k}g^{k1}$, we find that at $(\xi_0,t_0)$
$$\frac{d\widetilde{W}}{dt}=\frac{dh_1^1}{dt}$$
and the spatial derivatives also coincide. Replacing $\theta$ by $\widetilde{\theta}=log \widetilde{W}+p(u)+N\widetilde\rho$, we see that $\widetilde{\theta}$ attains a maximal value at $(\xi_0,t_0)$, where $\widetilde{W}$ satisfies the same differential equation in this point as $h_1^1$. Thus, without loss of generality, we may pretend $h_1^1$ to be a scalar and $\theta$ to be given by
\begin{equation}\label{3.16}
	\theta=log h_1^1+p(u)+N\rho.
\end{equation}
In order to calculate the evolution equations of $\theta$, we should deduce the evolution equations of $h_1^1,$ $u$ and $\rho$. By (\ref{3.3}), we have
\begin{equation}\label{3.17}
	\p_th_{ij}=-\nn_i\nn_j(u^\alpha f^{-\beta}-\eta u)+(u^\alpha f^{-\beta}-\eta u)(h^2)_{ij}-K(u^\alpha f^{-\beta}-\eta u)g_{ij}.
\end{equation}
Remark that $$G^{ij}=\frac{\p G}{\p h_{ij}},\qquad G^{ij,mn}=\frac{\p^2G}{\p h_{ij}\p h_{mn}}.$$
For convenience, we denote $F^{ij},F^{ij,mn}$ by $f^{ij},f^{ij,mn}$. Using (\ref{3.3}) and (\ref{3.17}), we have
\begin{equation*}
	\begin{split}
\frac{\p h_1^1}{\p t}=&g^{k1}\frac{\p h_{k1}}{\p t}-g^{1i}\frac{\p g_{ij}}{\p t}g^{jk}h_{1k}\\
=&-\nn_1\nn_1(\Psi G-\eta u)+(\Psi G-\eta u)h^2_{11}-K(\Psi G-\eta u)-2(\Psi G-\eta u)h^2_{11}\\
=&-\Psi_{;11}G-2\Psi_{;1}G_{;1}-\Psi G_{;11}+\eta u_{;11}-K(\Psi G-\eta u)-(\Psi G-\eta u)h^2_{11}\\
=&-(\a u^{\a-1}u_{;11}+\a(\a-1)u^{\a-2}u_{;1}^2)G-\Psi(G^{ij}h_{ij;11}+G^{ij,mn}h_{ij;1}h_{mn;1})\\
&-2\Psi_{;1}G_{;1}+\eta u_{;11}-K(\Psi G-\eta u)-(\Psi G-\eta u)h^2_{11}.
	\end{split}
\end{equation*}
Since $G^{ij,mn}=-\beta f^{-\beta-1}f^{ij,mn}+\beta(\beta+1)f^{-\beta-2}f^{ij}f^{mn}.$ And by (\ref{2.11}), (\ref{2.12}) and Ricci identity, we have
\begin{equation*}
	\begin{split}
\nn_1\nn_1h_{ij}= &h_{11;ij}+{R_{j11}}^ah_{ai}+{R_{j1i}}^ah_{a1}\\
= &h_{11;ij}+h_{ai}h_{aj}h_{11}+Kh_{ij}-h_{ij}h_{11}^2+K\delta_{1i}h_{1j}-K\delta_{1j}h_{1i}-K\delta_{ij}h_{11}.
	\end{split}
\end{equation*}
Thus
\begin{equation*}
	\begin{split}
\frac{\p h_1^1}{\p t}=&(\eta-\a u^{\a-1}G)u_{;11}-\a(\a-1)u^{\a-2}(\nn_1u)^2G\\
&-\Psi G^{ij}(h_{11;ij}+h_{ai}h_{aj}h_{11}+Kh_{ij}-h_{ij}h_{11}^2-K\delta_{ij}h_{11})\\
&-\Psi(-\beta f^{-\beta-1}f^{ij,mn}+\beta(\beta+1)f^{-\beta-2}f^{ij}f^{mn})h_{ij;1}h_{mn;1}\\
&+2\a\beta u^{\a-1}f^{-\beta-1}\nn_1u\nn_1f-K(\Psi G-\eta u)-(\Psi G-\eta u)h^2_{11}.
	\end{split}
\end{equation*}
Due to (\ref{2.13}), we have
\begin{equation*}
	\begin{split}
\frac{\p h_1^1}{\p t}=&\beta\Psi f^{-\beta-1}f^{ij}h_{11,ij}+\phi(\eta-\a u^{\a-1}G)\rho_{;k}{h_{11;}}^k+\phi'(\eta-\a u^{\a-1}G)h_{11}\\
&+(\a-\beta-1)\Psi Gh_{11}^2-\a(\a-1)\Psi G(\nn_1logu)^2-\Psi G^{ij}h_{ai}h_{aj}h_{11}\\
&-K(1-\beta)\Psi G+K\eta u+K\sum_{i}G^{ii}\Psi h_{11}+\beta\Psi F^{-\beta-1}F^{ij,mn}h_{ij;1}h_{mn;1}\\
&-\beta(\beta+1)\Psi G(\nn_1logf)^2+2\a\beta\Psi G\nn_1logu\nn_1logf.
	\end{split}
\end{equation*}
Define the operator $\mathcal{L}$ by
\begin{equation}\label{3.18}
	\mathcal{L}=\p_t-\beta\Psi f^{-\beta-1}f^{ij}\nn_{ij}^2-\phi(\eta-\a u^{\a-1}G)\rho_{;k}\nn^k.
\end{equation}
Since
\begin{equation}\label{3.19}
2\a\beta\Psi G\nn_1logu\nn_1logf\leq\beta(\beta+1)\Psi G(\nn_1logf)^2+\frac{\beta\a^2}{\beta+1}\Psi G(\nn_1logu)^2,
\end{equation}
we have
\begin{equation}\label{3.20}
	\begin{split}
\cL h_1^1\leq&\phi'(\eta-\a u^{\a-1}G)h_{11}+(\a-\beta-1)\Psi Gh_{11}^2+\frac{\a(\beta+1-\a)}{\beta+1}\Psi G(\nn_1logu)^2\\
&-\Psi G^{ij}h_{ai}h_{aj}h_{11}-K(1-\beta)\Psi G+K\eta u+K\sum_{i}G^{ii}\Psi h_{11}\\
&+\beta\Psi f^{-\beta-1}f^{ij,mn}h_{ij;1}h_{mn;1}.
	\end{split}
\end{equation}
In addition,
\begin{equation*}
	\begin{split}
\frac{\p u}{\p t}=&\frac{\p}{\p t}<V,\nu>\\
=&(\Psi G-\eta u)\phi'+(\eta-\a u^{\a-1}G)<V,\nn u>+\beta\Psi f^{-\beta-1}<V,\nn f>.
	\end{split}
\end{equation*}
By (\ref{2.13}), we deduce
\begin{equation}\label{3.21}
	\begin{split}
\cL u=((1-\beta)\Psi G-\eta u)\phi'+\beta u^{\a+1}f^{-\beta-1}f^{ij}(h^2)_{ij}.
	\end{split}
\end{equation}
By (\ref{2.16}) and (\ref{2.17}), we have
\begin{equation}\label{3.22}
	\begin{split}
\cL\rho=&(\Psi G-\eta u)\omega-\beta\frac{\phi'}{\phi}\Psi f^{-\beta-1}f^{ij}g_{ij}-\phi(\eta-\a u^{\a-1}G)\vert\nn\rho\vert^2\\
&+\beta\frac{\phi'}{\phi}\Psi f^{-\beta-1}f^{ij}\rho_{;i}\rho_{;j}+\frac{\beta\Psi G}{\omega}.
	\end{split}
\end{equation}
Note that $\a\leq0$, $\beta>0$, $\1-\a+\beta>2\beta>0$. If $\k_1$ is sufficiently large, the combination of (\ref{3.20}), (\ref{3.21}) and (\ref{3.22}) gives
\begin{equation}\label{3.23}
	\begin{split}
\cL\theta=&\frac{\cL h_1^1}{h_1^1}+\beta\Psi f^{-\beta-1}f^{ij}\nn_i(logh_1^1)\nn_j(logh_1^1)+p'\cL u-p''\beta\Psi f^{-\beta-1}f^{ij}\nn_iu\nn_ju\\
&+N\cL\rho\\
\leq&\phi'(\eta-\a u^{\a-1}G)+(\a-\beta-1)\Psi Gh_{11}+\beta\Psi f^{-\beta-1} f^{ij}(h^2)_{ij}+\frac{c}{\k_1}\\
&-K\beta f^{-\beta-1}\sum_{i}f^{ii}\Psi+\frac{\beta\Psi f^{-\beta-1}f^{ij,mn}h_{ij;1}h_{mn;1}}{h_{11}}\\
&+\beta\Psi f^{-\beta-1}f^{ij}(\nn_i(logh_1^1)\nn_j(logh_1^1)-p''\nn_iu\nn_ju)+p'\(((1-\beta)\Psi G-\eta u)\phi'\\
&+\beta u^{\a+1}f^{-\beta-1}f^{ij}(h^2)_{ij}\)+N\((\Psi G-\eta u)\omega-\beta\frac{\phi'}{\phi}\Psi f^{-\beta-1}f^{ij}g_{ij}\\
&-\phi(\eta-\a u^{\a-1}G)\vert\nn\rho\vert^2+\beta\frac{\phi'}{\phi}\Psi f^{-\beta-1}f^{ij}\rho_{;i}\rho_{;j}+\frac{\beta\Psi G}{\omega}\)\\
\leq&c+(\a-\beta-1)\Psi Gh_{11}+\beta(1+p'u)\Psi f^{-\beta-1} f^{ij}(h^2)_{ij}-K\beta f^{-\beta-1}f^{ij}g_{ij}\Psi\\
&+\frac{\beta\Psi f^{-\beta-1}f^{ij,mn}h_{ij;1}h_{mn;1}}{h_{11}}+\beta\Psi f^{-\beta-1}f^{ij}(\nn_i(logh_1^1)\nn_j(logh_1^1)-p''\nn_iu\nn_ju)\\
&-N\beta\frac{\phi'}{\phi}\Psi f^{-\beta-1}f^{ij}g_{ij}-N\phi(\eta-\a u^{\a-1}G)\vert\nn\rho\vert^2+N\beta\frac{\phi'}{\phi}\Psi f^{-\beta-1}f^{ij}\rho_{;i}\rho_{;j}.
\end{split}
\end{equation}
Due to the concavity of $f$ it holds that
\begin{equation}\label{3.24}
f^{kl,rs}\xi_{kl}\xi_{rs}\leq\sum_{k\neq l}\frac{f^{kk}-f^{ll}}{\k_k-\k_l}\xi_{kl}^2\leq\frac{2}{\k_1-\k_n}\sum_{k=1}^{n}(f^{11}-f^{kk})\xi_{1k}^2
\end{equation}
for all symmetric matrices $(\xi_{kl})$; cf. \cite{GC2}. Furthermore, we have
\begin{equation}\label{3.25}
f^{11}\leq\cdots\leq f^{nn};
\end{equation}
cf. \cite{EH}. In order to estimate (\ref{3.23}), we distinguish between two cases.

Case 1: $\k_n<-\eps_1\k_1$, $0<\eps_1<\frac{1}{2}$. Then
\begin{equation}\label{3.26}
f^{ij}(h^2)_{ij}\geq f^{nn}\k_n^2\geq\frac{1}{n}f^{ij}g_{ij}\k_n^2\geq\frac{1}{n}f^{ij}g_{ij}\eps_1^2\k_1^2.
\end{equation}
We use $\nn\theta=0$ to obtain
\begin{equation}\label{3.27}
f^{ij}\nn_i(logh_1^1)\nn_j(logh_1^1)=p'^2f^{ij}u_{;i}u_{;j}+2Np'f^{ij}u_{;i}\rho_{;j}+N^2f^{ij}\rho_{;i}\rho_{;j}.
\end{equation}
In this case, the concavity of $f$ implies that
\begin{equation}\label{3.28}
\frac{\beta\Psi f^{-\beta-1}f^{ij,mn}h_{ij;1}h_{mn;1}}{h_{11}}\leq0.
\end{equation}
By (\ref{2.13}) and note that $p'<0$, we have
\begin{equation}\label{3.29}
	\begin{split}
2N\beta p'\Psi f^{-\beta-1}f^{ij}u_{;i}\rho_{;j}=&2N\beta p'\phi\Psi f^{-\beta-1}f^{ij}\k_i\rho_{;i}\rho_{;j}.	
\end{split}
\end{equation}
By \cite{UJ} Lemma 3.3, we know $\frac{H}{n}\geq \frac{f}{\eta}\geq C_9$, where $H=\sum_{i}^{n}\k_i$ is the mean curvature. Thus $(n-1)\k_1+\k_n\ge H\ge C_9$. We can derive $\k_i\ge\k_n\ge C_9-(n-1)\k_1$. For fixed $i$, if $\k_i\geq0$, we derive $$2N\beta p'\phi\Psi f^{-\beta-1}f^{ij}\k_i\rho_{;i}\rho_{;j}\leq0.$$ If $\k_i<0$, we have
$$ 2N\beta p'\phi\Psi f^{-\beta-1}f^{ij}\k_i\rho_{;i}\rho_{;j}\leq2N\beta p'\phi\Psi f^{-\beta-1}\k_if^{ij}g_{ij} \leq C_{10}(C_9-(n-1)\k_1)f^{ij}g_{ij},$$
 where we have used $f^{ij}\rho_{;i}\rho_{;j}\leq f^{ij}g_{ij}\vert\nn\rho\vert^2\leq f^{ij}g_{ij}$.

Without loss of generality, we can assume that $\k_k\ge0$ and $\k_{k+1}\le0$, then
\begin{equation}\label{3.30}
2N\beta p'\phi\Psi f^{-\beta-1}f^{ij}\k_i\rho_{;i}\rho_{;j}\leq(n-k)C_{10}(C_9-(n-1)\k_1)f^{ij}g_{ij}.
\end{equation}

Since $p'^2=p''$ and $1+p'u<0$, by the combination of (\ref{2.19}), (\ref{3.23}), (\ref{3.26}), (\ref{3.27}), (\ref{3.28}) and (\ref{3.30}), in this case (\ref{3.23}) becomes
\begin{equation}\label{3.31}
	\begin{split}
\cL\theta\leq\beta f^{-\beta-1}\Psi f^{ij}g_{ij}\(\frac{1}{n}\eps_1^2\k_1^2(1+p'u)+C_{11}\k_1+C_{12}\)+(\a-\beta-1)\Psi G\k_1+C_{13},
	\end{split}
	\end{equation}
which is negative for large $\k_1$. We also use $\a-\beta-1<-2\beta<0$ in there.

Case 2: $\k_n\geq-\eps_1\k_1$. Then
$$\frac{2}{\k_1-\k_n}\sum_{k=1}^{n}(f^{11}-f^{kk})(h_{11;k})^2k_1^{-1}\leq\frac{2}{1+\eps_1}\sum_{k=1}^{n}(f^{11}-f^{kk})(h_{11;k})^2k_1^{-2}.$$
We deduce further
\begin{equation}\label{3.32}
	\begin{split}
f^{ij}&\nn_i(logh_1^1)\nn_j(logh_1^1)+\frac{2}{\k_1-\k_n}\sum_{k=1}^{n}(f^{11}-f^{kk})(h_{11;k})^2k_1^{-1}\\
\leq&\frac{2}{1+\eps_1}\sum_{k=1}^{n}f^{11}(logh_1^1)_{;k}^2-\frac{1-\eps_1}{1+\eps_1}\sum_{k=1}^{n}f^{kk}(logh_1^1)_{;k}^2\\
\leq&\sum_{k=1}^{n}f^{11}(logh_1^1)_{;k}^2\\
\leq&f^{11}(p'^2\vert\nn u\vert^2+2Np'<\nn u,\nn\rho>+N^2\vert\nn\rho\vert^2),
	\end{split}
\end{equation}
where we have used $f^{kk}\geq f^{11}$ in the second inequality.
Note that
\begin{gather}
\beta(1+p'u)\Psi f^{-\beta-1}f^{ij}(h^2)_{ij}\leq\beta(1+p'u)\Psi f^{-\beta-1}f^{11}\k_1^2,\label{3.33}\\
2Np'f^{11}<\nn u,\nn\rho>=2Np'\phi f^{11}\k_i\rho_{;i}^2\leq-2\eps_1Np'\phi f^{11}\rho_{;i}^2\k_1,\label{3.34}\\
-\beta\Psi f^{-\beta-1}p''f^{ij}\nn_iu\nn_ju+\beta\Psi f^{-\beta-1}p''f^{11}\vert\nn u\vert^2\leq0,\label{3.35}\\
N\beta\frac{\phi'}{\phi}\Psi f^{-\beta-1}f^{ij}\rho_{;i}\rho_{;j}\leq N\beta\frac{\phi'}{\phi}\Psi f^{-\beta-1}f^{ij}g_{ij}\vert\nn\rho\vert^2.\label{3.36}
\end{gather}
According to (\ref{3.36}), we have
\begin{equation}\label{3.37}
-N\beta\frac{\phi'}{\phi}\Psi f^{-\beta-1}f^{ij}g_{ij}+N\beta\frac{\phi'}{\phi}\Psi f^{-\beta-1}f^{ij}\rho_{;i}\rho_{;j}\leq-N\beta\frac{\phi'}{\phi}\Psi f^{-\beta-1}f^{ij}g_{ij}\frac{1}{\omega^2}.
\end{equation}
By the combination of (\ref{3.32}), (\ref{3.33}), (\ref{3.34}), (\ref{3.35}), (\ref{3.36}) and (\ref{3.37}), (\ref{3.23}) becomes
\begin{equation}\label{3.38}
\begin{split}
\cL\theta\leq& c+(\a-\beta-1)\Psi G\k_1+\beta\Psi f^{-\beta-1}f^{11}\((1+p'u)k_1^2+C_{14}\k_1+C_{15}\)\\
&-\beta f^{-\beta-1}f^{ij}g_{ij}\Psi(N\frac{\phi'}{\phi\omega^2}+K),
\end{split}
\end{equation}
which is negative for large $\k_1$ after fixing $N_0$ large enough to ensure that
$$N_0\frac{\phi'}{\phi\omega^2}+K\geq0.$$
Hence in this case any $N\geq N_0$ yields an upper bound for $\k_1$.

In conclusion, $\k_i\leq C_7$, where $C_7$ depends on the initial hypersurface, $\a$ and $\beta$. Since $f$ is uniformly continuous on the convex cone $\overline{\Gamma}$, and $f$ is bounded from below by a positive constant. By Corollary \ref{c3.6} and Assumption \ref{a1.1} imply that $\k_i$ remains in a fixed compact subset of $\Gamma$, which is independent of $t$.
\end{proof}

The estimates obtained in Lemma \ref{l3.1}, \ref{l3.5}, \ref{l3.7} and Corollary \ref{c3.6} depend on $\a$, $\beta$ and the geometry of the initial data $M_0$. They are independent of $T$. By Lemma \ref{l3.1}, \ref{l3.5}, \ref{l3.7} and Corollary \ref{c3.6}, we conclude that the equation (\ref{2.17}) is uniformly parabolic. By the $C^0$ estimate (Lemma \ref{l3.1}), the gradient estimate (Lemma \ref{l3.5}), the $C^2$ estimate (Lemma \ref{l3.7}) and the Krylov's and Nirenberg's theory \cite{KNV,LN}, we get the H$\ddot{o}$lder continuity of $\nabla^2\rho$ and $\rho_t$. Then we can get higher order derivation estimates by the regularity theory of the uniformly parabolic equations. Hence we obtain the long time existence and $C^\infty$-smoothness of solutions for the expanding flow (\ref{1.7}). The uniqueness of smooth solutions also follows from the parabolic theory. In summary, we have proved the following theorem.
\begin{theorem}\label{t3.8}
Let $M_0$ be a smooth, closed and star-shaped hypersurface in $\mathbb{K}^{n+1}$, $n\geq2$, which encloses the origin. If (\romannumeral1) $\alpha\leq0<\beta<1-\alpha$ in $\mH^{n+1}$ or (\romannumeral2) $\alpha\leq0<\beta\le1-\alpha$ in $\mR^{n+1}$, the expanding flow (\ref{1.7}) has a unique smooth, closed and star-shaped solution $M_t$ for all time $t\geq0$. Moreover, the radial function of $M_t$ satisfies the a priori estimates
$$\parallel\rho\parallel_{C^{k,\beta}(\mS^n\times[0,\infty))}\leq C,$$
where the constant $C>0$ depends only on $k,\a,\beta$ and the geometry of $M_0$.
\end{theorem}

\section{Proof Of Theorem 1.2 and 1.3}

In this section, we prove the asymptotical convergence of solutions to the expanding flow (\ref{1.7}). By Theorem \ref{t3.8} it is known that the flow (\ref{1.7}) exists for all time $t>0$ and remains smooth and star-shaped, provided $M_0$ is smooth, star-shaped and encloses the origin. In Section 3, we have the bound of $\rho$, $\vert D\g\vert$ and $f$. It then follows by (\ref{3.9}) that $\p_tO_{\max}\leq-C_0O_{\max}$ for some positive constant $C_0$, where $O=\frac{1}{2}\vert D\g\vert^2$. This proves
\begin{equation}\label{4.1}
\max_{\mS^n}\vert D\g\vert^2\leq Ce^{-C_0t},\forall t>0,
\end{equation}
for both $C$ and $C_0$ are positive constants. Meanwhile, according to the bound of $\phi$, we can derive
\begin{equation}\label{4.2}
\max_{\mS^n}\vert D\rho\vert^2\leq C'e^{-C_0t},\forall t>0.
\end{equation}

\begin{proof of theorem 1.2 and 1.3}

By (\ref{4.2}), we have that $\vert D\rho\vert\to0$ exponentially as $t\to\infty$. Hence by the interpolation and the a priori estimates, we can get that $\rho$ converges exponentially to a constant in the $C^\infty$ topology as $t\to\infty$.

\end{proof of theorem 1.2 and 1.3}

\section{Proof Of Theorem 1.4 and 1.5}
The aim of this section is to proof some monotone quantities involving a weighted $\sigma_k$ or power of $\sigma_k$ integral along inverse curvature flows in the Euclidean space $\mR^{n+1}$. In this section we use $\widetilde{S}_{\iota,k}$, $\widetilde{T}_{\iota,k}$, etc. to express the quantities along the flow (\ref{1.6}) or (\ref{1.7}), and $S_{\iota,k}$, $T_{\iota,k}$, etc. express the quantities evolving by the flow (\ref{1.1}).

First, we derive the evolution equation of $\widetilde{S}_{\iota,k}$ and $\widetilde{T}_{\iota,k}$. In the rest of this paper, without loss of generality, we can let $f^{-\beta}(1,\cdots,1)=1$, i.e. $\eta=1$.

\begin{lemma}\label{l5.1}
If $\a=1-\beta$, Denote $\mathscr{F}=u^\a f^{-\beta}$. Under the flow (1.6) or (1.7) in $\mR^{n+1}$, we have
\begin{align}
\p_t\widetilde{S}_{\iota,k}=&e^{-(n-k+\iota)t}\int_{M_t}\iota(n-k+1)u^{\iota-1}p_k\mathscr{F}+(1-\iota)(n-k)u^\iota p_{k+1}\mathscr{F}\notag\\
&-(n-k+\iota)u^\iota p_k+\iota(\iota-1)u^{\iota-2}\mathscr{F}(<\nabla u,\nabla\Phi>p_k-p_k^{ij}\nabla^iu\nabla_ju)d\mu,\label{5.1}\\
\p_t\widetilde{T}_{\iota,k}=&e^{-(n-k\iota)t}\int_{M_t}-(n-k\iota)p_k^\iota+\iota(\iota-1)p_k^{\iota-2}p_k^{ij}\nabla^ip_k\nabla_j\mathscr{F}+(1-\iota)p_k^\iota\mathscr{F}H\notag\\
&+(n-k)\iota p_k^{\iota-1}p_{k+1}\mathscr{F}d\mu,\label{5.2}
\end{align}
where $H=np_1$ is the mean curvature.
\end{lemma}
\begin{proof}
Note that $\widetilde{X}=e^{-t}X$. It is easy to find that $\widetilde{S}_{\iota,k}=e^{-(n-k+\iota)t}S_{\iota,k}$ and $\widetilde{T}_{\iota,k}=e^{-(n-k\iota)t}T_{\iota,k}$. Then by (\ref{3.3}), we have
\begin{align*}
\p_t\widetilde{S}_{\iota,k}=&\p_t(e^{-(n-k+\iota)t}S_{\iota,k})\\
=&e^{-(n-k+\iota)t}\int_{M_t}-(n-k+\iota)u^\iota p_k+\iota u^{\iota-1}p_k(\mathscr{F}-<\nabla\mathscr{F},\nabla\Phi>)\\
&+u^\iota p_k^{ij}(-\nn_i\nn^j\mathscr{F}-\mathscr{F}(h^2)_i^j)+u^\iota p_kH\mathscr{F}d\mu.
\end{align*}
By (\ref{2.13}) and integral by part, we have
\begin{equation}\label{5.3}
\begin{split}
\int_{M_t}u^\iota p_k^{ij}\nn_i\nn^j\mathscr{F}d\mu=&-\int_{M_t}\iota u^{\iota-1}p_k^{ij}\nn_iu\nn^j\mathscr{F}d\mu\\
=&\int_{M_t}\mathscr{F}(\iota(\iota-1)u^{\iota-2}p_k^{ij}\nn_iu\nn^ju+\iota u^{\iota-1}p_k^{ij}\nn^j\nn_iu)d\mu\\
=&\int_{M_t}\mathscr{F}\(\iota(\iota-1)u^{\iota-2}p_k^{ij}\nn_iu\nn^ju\\
&+\iota u^{\iota-1}(<\nn\Phi,\nn p_k>+kp_k-u(Hp_k-(n-k)p_{k+1}))\)d\mu,
\end{split}
\end{equation}
where we use $p_k^{ij}(h^2)_i^j=Hp_k-(n-k)p_{k+1}$ and $\sum_{i}\nn_ip_k^{ij}=0$. By (\ref{5.3}) and $\nn_i\nn_j\Phi=g_{ij}\phi'-h_{ij}u$, we have
\begin{equation*}
\begin{split}
\p_t\widetilde{S}_{\iota,k}=&e^{-(n-k+\iota)t}\int_{M_t}-(n-k+\iota)u^\iota p_k+(1-k)\iota u^{\iota-1}p_k\mathscr{F}+\iota u^\iota Hp_k\mathscr{F}\\
&-\iota u^{\iota-1}<\nn\Phi,\nn(p_k\mathscr{F})>+(1-\iota)(n-k)u^\iota p_{k+1}\mathscr{F}\\
&-\iota(\iota-1)u^{\iota-2}\mathscr{F}p_k^{ij}\nn_iu\nn^jud\mu\\
=&e^{-(n-k+\iota)t}\int_{M_t}-(n-k+\iota)u^\iota p_k+(1-k)\iota u^{\iota-1}p_k\mathscr{F}+\iota u^\iota Hp_k\mathscr{F}\\
&+\iota(\iota-1)u^{\iota-2}p_k\mathscr{F}<\nn\Phi,\nn u>+\iota u^{\iota-1}(n-Hu)p_k\mathscr{F}\\
&+(1-\iota)(n-k)u^\iota p_{k+1}\mathscr{F}-\iota(\iota-1)u^{\iota-2}\mathscr{F}p_k^{ij}\nn_iu\nn^jud\mu\\
=&e^{-(n-k+\iota)t}\int_{M_t}\iota(n-k+1)u^{\iota-1}p_k\mathscr{F}+(1-\iota)(n-k)u^\iota p_{k+1}\mathscr{F}\\
&-(n-k+\iota)u^\iota p_k+\iota(\iota-1)u^{\iota-2}\mathscr{F}(<\nabla u,\nabla\Phi>p_k-p_k^{ij}\nabla^iu\nabla_ju)d\mu.
\end{split}
\end{equation*}
Similarly,
\begin{equation*}
\begin{split}
\p_t\widetilde{T}_{\iota,k}=&\p_t(e^{-(n-k\iota)t}T_{\iota,k})\\
=&e^{-(n-k\iota)t}\int_{M_t}-(n-k\iota)p_k^\iota+\iota p_k^{\iota-1}p_k^{ij}(-\nn^i\nn_j\mathscr{F}-\mathscr{F}(h^2)^i_j)+p_k^\iota\mathscr{F}Hd\mu\\
=&e^{-(n-k\iota)t}\int_{M_t}-(n-k\iota)p_k^\iota+\iota(\iota-1)p_k^{\iota-2}p_k^{ij}\nabla^ip_k\nabla_j\mathscr{F}+(1-\iota)p_k^\iota\mathscr{F}H\\
&+(n-k)\iota p_k^{\iota-1}p_{k+1}\mathscr{F}d\mu.
\end{split}
\end{equation*}
\end{proof}
Due to the evolution equations in Lemma \ref{l5.1}, we can prove Theorem \ref{t1.4} and \ref{t1.5}.

\begin{proof of theorem 1.4}
	
Let $\a=0$, $\beta=1$ and $\mathscr{F}=\frac{p_{k-1}}{p_k}$. Denote
\begin{equation*}
\begin{split}
\Rmnum1=&\int_{M_t}\iota(n-k+1)u^{\iota-1}p_k\mathscr{F}+(1-\iota)(n-k)u^\iota p_{k+1}\mathscr{F}-(n-k+\iota)u^\iota p_k\\
&+\iota(\iota-1)u^{\iota-2}\mathscr{F}(<\nabla u,\nabla\Phi>p_k-p_k^{ij}\nabla^iu\nabla_ju)d\mu.
\end{split}
\end{equation*}
Then the monotonicity of $\widetilde{S}_{\iota,k}$ is equivalent to the positivity of $\Rmnum1$.

 We first consider the case $(\rmnum5)$.
By the bound of $k$, we can derive $n-k+1\ge0$ and $n-k\ge0$. If $\iota<0$ we have $\iota(\iota-1)>0$. By Newton-Maclaurin inequality, we can estimate $\Rmnum1$.
\begin{equation}\label{5.4}
\begin{split}
\Rmnum1\le&\int_{M_t}\iota(n-k+1)u^{\iota-1}(p_{k-1}-up_k)\\
&+\iota(\iota-1)u^{\iota-2}\(<\nabla u,\nabla\Phi>p_{k-1}-\frac{p_{k-1}}{p_k}p_k^{ij}\nabla_iu\nabla^ju\)d\mu\\
\le&\int_{M_t}\frac{\iota(n-k+1)}{k}u^{\iota-1}p_k^{ij}\nn_i\nn^j\Phi\\
&+\iota(\iota-1)u^{\iota-2}\(<\nabla u,\nabla\Phi>p_{k-1}-\frac{p_{k-1}}{p_k}p_k^{ij}\nabla_iu\nabla^ju\)d\mu\\
\le&\int_{M_t}\iota(\iota-1)u^{\iota-2}\(-\frac{n-k+1}{k}p_k^{ij}\nn_iu\nn^j\Phi+<\nabla u,\nabla\Phi>p_{k-1}\\
&-\frac{p_{k-1}}{p_k}p_k^{ij}\nabla_iu\nabla^ju\)d\mu.
\end{split}
\end{equation}
Denote
$$\Rmnum2=-\frac{n-k+1}{k}p_k^{ij}\nn_iu\nn^j\Phi+<\nabla u,\nabla\Phi>p_{k-1}-\frac{p_{k-1}}{p_k}p_k^{ij}\nabla_iu\nabla^ju.$$
The choice of coordinate does not change the sign of $\Rmnum2$. So we can choose coordinates in $(x_0,t_0)$, such that in this point $p_k^{ij}$, $g_{ij}$ and ${h_i^j}$ is diagonal. In other words, $h_i^j=\k_i\delta_{ij}$. By Lemma \ref{l2.2}, we have $\nn_iu=\k_i\nn_i\Phi$. Note that $\frac{C_n^k}{C_n^{k-1}}=\frac{n-k+1}{k}$, where $C_n^k=\frac{n!}{k!(n-k)!}$. We can derive
\begin{equation*}
\begin{split}
C_n^{k-1}\Rmnum2=&-\sigma_k^{ii}\nn_iu\nn^i\Phi+\sigma_{k-1}\nn_iu\nn^i\Phi-\frac{\sigma_{k-1}}{\sigma_k}\sigma_k^{ii}\nabla_iu\nabla^iu\\
=&(\sigma_{k-1}-\sigma_k^{ii})\nn_iu\nn^i\Phi-\frac{\sigma_{k-1}}{\sigma_k}\sigma_k^{ii}\nabla_iu\nabla^iu\\
=&\sigma_{k-2,i}\k_i\nn_iu\nn^i\Phi-\frac{\sigma_{k-1}}{\sigma_k}\sigma_k^{ii}\nabla_iu\nabla^iu\\
=&(\sigma_{k-2,i}-\frac{\sigma_{k-1}}{\sigma_k}\sigma_k^{ii})(\nabla_iu)^2\\
=&\frac{1}{\sigma_k}(\sigma_{k-1}^{ii}\sigma_k-\sigma_{k-1}\sigma_k^{ii})(\nabla_iu)^2.
\end{split}
\end{equation*}
By \cite{HGC}, $\frac{\sigma_k}{\sigma_{k-1}}$ is increasing. This also explains $\sigma_{k-1}^{ii}\sigma_k-\sigma_{k-1}\sigma_k^{ii}\le0$. From the above we can get $\p_t\widetilde{S}_{\iota,k}\le0$. If $M_t$ is a round sphere, $\nn_iu=0$ for $\forall i$. Thus, $\widetilde{S}_{\iota,k}$ is invariant if and only if $M_t$ is a round sphere for each $t$. We complete the proof of the case ($\rmnum5$).

The case ($\rmnum2$) and ($\rmnum4$) is obvious by (\ref{5.1}) and Minkowski formula $\int_{M}up_kd\mu=\int_{M}\phi'p_{k-1}d\mu$.

The proof of the case ($\rmnum1$) and ($\rmnum3$) is similar to the case ($\rmnum5$).
\end{proof of theorem 1.4}

Remark: In case ($\rmnum1$) and ($\rmnum3$), we only prove the situation of $k=n$. In fact, if $k<n$, there is no situation where $\Rmnum2$ and the second term of $\Rmnum1$  have upper or lower bounds at the same time. We also consider the monotonicity $\widetilde{S}_{\iota,l}$ along the flow (\ref{1.6}) or (\ref{1.7}) with $\mathscr{F}=p_{k-1}/p_k$ for $l<k$. If $0<\iota<1$, the first four terms of $\Rmnum1$ have lower bound at the same time, but whether ``$\Rmnum2$" after the same treatment can be estimated is a question.\\

\begin{proof of theorem 1.5}

Let $\a=0$, $\beta=1$ and $\mathscr{F}=p_k^{\frac{1}{k}}$. Denote
\begin{equation*}
	\begin{split}
		\Rmnum3=&\int_{M_t}-(n-k\iota)p_k^\iota+\iota(\iota-1)p_k^{\iota-2}p_k^{ij}\nabla^ip_k\nabla_j\mathscr{F}+(1-\iota)p_k^\iota\mathscr{F}H\\
		&+(n-k)\iota p_k^{\iota-1}p_{k+1}\mathscr{F}d\mu.
	\end{split}
\end{equation*}
Then the monotonicity of $\widetilde{T}_{\iota,k}$ is equivalent to the positivity of $\Rmnum3$.

We first consider the case ($\rmnum1$). By the bound of $k$, we can derive $n-k$. If $\iota>1$ we have $\iota(\iota-1)>0$ and $1-\iota<0$. By Newton-Maclaurin inequality and the Positive definiteness of $p_k^{ij}$, we can estimate $\Rmnum3$.
\begin{equation*}
\begin{split}
\Rmnum3=&\int_{M_t}-(n-k\iota)p_k^\iota-\frac{\iota(\iota-1)}{k}p_k^{\iota-3-\frac{1}{k}}p_k^{ij}\nabla^ip_k\nabla_jp_k+n(1-\iota)p_k^{\iota-\frac{1}{k}}p_1\\
&+(n-k)\iota p_k^{\iota-1-\frac{1}{k}}p_{k+1}d\mu\\
\le&\int_{M_t}-(n-k\iota)p_k^\iota+n(1-\iota)p_k^\iota+\iota(n-k)p_k^\iota d\mu=0.
\end{split}
\end{equation*}
This means $\p_t\widetilde{T}_{\iota,k}\le0$. If $M_t$ is a round sphere, $\nn_ip_k=0$ for $\forall i$. And the equality holds in Newton-MacLaurin inequality if and only if $M$ is a sphere. We can deduce $\widetilde{T}_{\iota,k}$ is invariant if and only if $M_t$ is a round sphere for each $t$. We have a complete proof of the case ($\rmnum1$).

The case ($\rmnum2$), ($\rmnum4$) and ($\rmnum5$) is obvious by (\ref{5.2}), Lemma \ref{l2.5} and Minkowski formula.

The proof of the case ($\rmnum3$) is similar to the case ($\rmnum1$).
\end{proof of theorem 1.5}

\section{Proof Of Theorem 1.7}
Along the flow $\frac{\p X}{\p t}=\(\frac{p_{n-m-1}}{u^mp_n}-u\)\nu$, we can deduce that
\begin{equation}\label{6.1}
\begin{split}
\frac{\p}{\p t}\int_{M_t}p_kd\mu=&(n-k)\int_{M_t}p_{k+1}\(\frac{p_{n-m-1}}{u^mp_n}-u\)d\mu\\
\geq&(n-k)\int_{M_t}\frac{p_{k-m}}{u^m}-p_kd\mu,
\end{split}
\end{equation}
where $0\le m\le k\le n-1$.

By \cite{KK} Proposition 4.3, we have
\begin{equation}\label{6.2}
\int f(u)p_k=\int uf(u)p_{k+1}-\frac{1}{(n-k)C_{n}^k}\int f'(u)<T_kA^\nu(Y^T),Y^T>,
\end{equation}
where $f$ is a smooth function on $\mR$, $Y$ is the position vector and $T_k$ is the $k$th Newton transformation.

We can derive $\int\frac{p_{k-m}}{u^m}d\mu\ge\int\frac{p_{k-m+1}}{u^{m-1}}d\mu\ge\cdots\ge\int p_kd\mu$ in convex hypersurface. This means that $\frac{\p}{\p t}\int_{M_t}p_kd\mu\ge0$. The proof of $(\rmnum1)$ has been completed.

To prove $(\rmnum2)$ and $(\rmnum3)$, we still use the notations in Section 5. By (\ref{5.1}), we have
\begin{equation}\label{6.3}
	\begin{split}
\p_t\widetilde{S}_{\iota,n}=&e^{-\iota t}\int_{M_t}\iota u^{\iota-1}\frac{p_{n-m-1}}{u^m}-\iota u^\iota p_n\\
&+\iota(\iota-1)u^{\iota-2}\frac{p_{n-m-1}}{up_n}(<\nabla u,\nabla\Phi>p_n-p_n^{ij}\nabla^iu\nabla_ju)d\mu.
	\end{split}
\end{equation}
Choose appropriate coordinates to make $[p_n^{ij}]$ diagonal. We can get
$$<\nabla u,\nabla\Phi>p_n-p_n^{ij}\nabla^iu\nabla_ju=p_n\nn_iu\nn_i\Phi-p_n^{ij}\k_i\nn_iu\nn_i\Phi=0.$$
Thus, \begin{equation}\label{6.4}
\p_t\widetilde{S}_{\iota,n}=\iota e^{-\iota t}\int_{M_t} u^{\iota-1-m}p_{n-m-1}-u^\iota p_nd\mu.
\end{equation}
By (\ref{6.2}), if $\iota-1-m\ge0,\cdots,\iota-1\ge0$, we have $\int u^{\iota-1-m}p_{n-m-1}\le\int u^{\iota-m}p_{n-m}\le\cdots\le\int u^\iota p_n.$ In other words, if $\iota\ge m+1$, we have $\int u^{\iota-1-m}p_{n-m-1}\le\int u^{\iota-m}p_{n-m}\le\cdots\le\int u^\iota p_n.$ This means $\p_t\widetilde{S}_{\iota,n}\le0.$

 Similarly, if $\iota\le1$, we have $\int u^{\iota-1-m}p_{n-m-1}\ge\int u^{\iota-m}p_{n-m}\ge\cdots\ge\int u^\iota p_n.$ This means $\p_t\widetilde{S}_{\iota,n}\ge0$ for $0<\iota\le1$ and $\p_t\widetilde{S}_{\iota,n}\le0$ for $\iota<0$. The proof of $(\rmnum2)$ and $(\rmnum3)$ has been completed.

The inequalities in Theorem \ref{t1.7} is a direct corollary of these monotone quantities. The proof of this inequalities is similar to Theorem \ref{t1.6} or Corollary \ref{c7.2}.

\section{Some Applications and Geometric inequalities}

In this section, we give a new proof of a family of inequalities involving the weighted integral of $k$th elementary symmetric function for $k$-convex, star-shaped hypersurfaces.

By Theorem \ref{t1.4} and \ref{t1.5}, we can derive a family of straightforward corollaries. The first corollary is Theorem \ref{t1.6}.

\begin{proof of theorem 1.6}
	
By the case ($\rmnum4$) of Theorem \ref{t1.4}, we have
$$\frac{d}{dt}S_{0,k-1}(t)=0 \text{\qquad   and   \qquad} \frac{d}{dt}S_{0,k}(t)\le0$$
under the flow $\frac{\p X}{\p t}=\(\frac{p_{k-1}}{p_k}-u\)\nu$. Theorem \ref{t1.2} says that the flow converges to some geodesic ball $B_r$ with $S_{0,k-1}(B_r)=S_{0,k-1}(0)=S_{0,k-1}(t)$, where we also denote $S_{\iota,k}(B_r)$ by $\int_{B_r}u^\iota p_kd\mu$. Thus we have
$$S_{0,k}(t)\ge S_{0,k}(B_r), \text{\quad with\quad }S_{0,k-1}(t)=S_{0,k-1}(B_r) \text{ for some } r>0,$$
which is equivalent to
$$\left(\frac{V_{(n+1)-k}(\Omega)}{V_{(n+1)-k}(B)}\right)^{\frac{1}{n+1-k}}= r\le\left(\frac{V_{n-k}(\Omega)}{V_{n-k}(B)}\right)^{\frac{1}{n-k}}$$
by $S_{\iota,k}(B_r)=\omega_nr^{n+\iota-k}$. Equality holds if and only if $S_{0,k}$ is a constant function. Namely, equality holds if and only if $M$ is a round sphere.
\end{proof of theorem 1.6}

We continue to list some direct corollaries of Theorem \ref{t1.4} and \ref{t1.5}, which may not be optimal.

\begin{corollary}\label{c7.2}
Suppose $M$ is a smooth, closed, star-shaped and $k$-convex hypersurface in $\mR^{n+1}$ for some $1\le k\le n$. Then
\begin{align}
\(\int_{M}u^\iota p_nd\mu\)^\frac{1}{\iota}&\ge\omega_n^{\frac{1}{\iota}-1}\int_{M}p_{n-1}d\mu,\quad k=n,\iota\ge1.\\
\(\int_{M}u^\iota p_nd\mu\)^\frac{1}{\iota}&\le\omega_n^{\frac{1}{\iota}-1}\int_{M}p_{n-1}d\mu,\quad k=n,0<\iota\le1.\\
\int_{M}p_ld\mu&\le\omega_n^{\frac{l-k}{n-k}}\(\int_{M}p_kd\mu\)^\frac{n-l}{n-k},\quad 1\le l\le k\le n.\\
\int_{M}u^\iota p_kd\mu&\ge\omega_n^{\frac{1-\iota}{n-k+1}}(\int_{M}p_{k-1}d\mu)^{\frac{n+\iota-k}{n-k+1}},\quad 1\le k\le n, \iota\le 0.\label{7.4}\\
\int_{M}p_k^\iota d\mu&\ge A(M)^{\frac{n-k\iota}{n}}\omega_n^{\frac{k\iota}{n}}, \quad 1\le k\le n, \iota\ge 1.\\
\(\int_{M}p_n^\iota d\mu\)^\frac{1}{n-n\iota}&\le\omega_n^{\frac{1-n+n\iota}{n-n\iota}}\int_{M}p_{n-1}d\mu,\quad k=n,0\le\iota<1,
\end{align}
where $A(M)$ is the area of $M$ and $\omega_n$ is the area of the unit sphere $\mS^n$ in $\mR^{n+1}$. The equalities of the above inequalities hold if and only if $M$ is a round sphere.
\end{corollary}
\begin{proof}
The proof of Corollary \ref{c7.2} is similar to Theorem \ref{t1.6}. We only prove (\ref{7.4}) here.

By Theorem \ref{t1.4}, we have
$$\frac{d}{dt}S_{0,k-1}(t)=0 \text{\qquad   and   \qquad} \frac{d}{dt}S_{\iota,k}(t)\le0$$
under the flow $\frac{\p X}{\p t}=\(\frac{p_{k-1}}{p_k}-u\)\nu$, where $\iota\le0$. Theorem \ref{t1.2} says that the flow converges to some geodesic ball $B_r$ with $S_{0,k-1}(B_r)=S_{0,k-1}(0)=S_{0,k-1}(t)$. Thus we have
$$S_{\iota,k}(t)\ge S_{\iota,k}(B_r), \text{\quad with\quad }S_{0,k-1}(t)=S_{0,k-1}(B_r) \text{ for some } r>0,$$
which is equivalent to
$$\int_{M}u^\iota p_kd\mu\ge\omega_nr^{n+\iota-k}= \omega_n^{\frac{1-\iota}{n-k+1}}(\int_{M}p_{k-1}d\mu)^{\frac{n+\iota-k}{n-k+1}},$$
by $S_{\iota,k}(B_r)=\omega_nr^{n+\iota-k}$. Equality holds if and only if $S_{\iota,k}$ is a constant function. Namely, equality holds if and only if $M$ is a round sphere.
\end{proof}

By the H$\ddot{o}$lder inequality, Minkowski formula and Corollary \ref{c7.2}, we can also get some inequalities.

\begin{corollary}\label{c7.3}
Suppose $M$ is a smooth, closed, star-shaped and $k$-convex hypersurface in $\mR^{n+1}$ for some $0\le k\le n$. Then
\begin{align}
\int_{M}u^\iota p_kd\mu&\ge\(\int_{M}p_{k-1}d\mu\)^{\iota}\(\int_{M}p_{k}d\mu\)^{1-\iota}, \quad \iota\ge1 \text{ or } \iota\le0.\label{7.8}\\
\int_{M}u^\iota p_kd\mu&\le\(\int_{M}p_{k-1}d\mu\)^{\iota}\(\int_{M}p_{k}d\mu\)^{1-\iota}, \quad 0\le\iota\le1.\label{7.9}\\
\int_Mp_k^\iota d\mu&\ge\(\int_{M}p_{k}d\mu\)^{\iota}A(M)^{1-\iota}, \quad \iota\ge1 \text{ or } \iota\le0.\label{7.10}\\
\int_Mp_k^\iota d\mu&\le\(\int_{M}p_{k}d\mu\)^{\iota}A(M)^{1-\iota}, \quad 0\le\iota\le1.\label{7.11}\\
\int_{M}up_k^\iota d\mu&\ge\(\int_{M}p_{k-1}d\mu\)^\iota\((n+1)V(\Omega)\)^{1-\iota}, \quad \iota\ge1.\label{7.12}\\
\int_{M}up_k^\iota d\mu&\le\(\int_{M}p_{k-1}d\mu\)^\iota\((n+1)V(\Omega)\)^{1-\iota}, \quad 0\le\iota\le1 \text{ or } \iota\le0\label{7.13}.
\end{align}
where we denote $p_{-1}=u$ and $\Omega$ is enclosed by $M$. The equalities of the above inequalities hold if and only if $M$ is a round sphere.
\end{corollary}
\begin{proof}
By the H$\ddot{o}$lder inequality, we can derive
\begin{align}
\int_{M}up_kd\mu&\le\(\int_{M}u^\iota p_kd\mu\)^\frac{1}{\iota}\(\int_{M}p_kd\mu\)^{1-\frac{1}{\iota}}, \quad 0\le\iota\le1.\\
\int_{M}p_kd\mu&\le\(\int_{M}up_kd\mu\)^\frac{-\iota}{1-\iota}\(\int_{M}u^\iota p_kd\mu\)^{\frac{1}{1-\iota}}, \quad \iota\le0.\\
\int_{M}u^\iota p_kd\mu&\le\(\int_{M}up_kd\mu\)^{\iota}\(\int_{M}p_kd\mu\)^{1-{\iota}},  \quad 0\le\iota\le1.\\
\int_{M}p_kd\mu&\le\(\int_{M}p_k^\iota d\mu\)^\frac{1}{\iota}A(M)^{\frac{\iota-1}{\iota}}, \quad \iota\ge1.\label{7.17}\\
\int_{M}p_k^\iota d\mu&\le\(\int_{M}p_k d\mu\)^{\iota}A(M)^{1-\iota}, \quad 0\le\iota\le1.
\end{align}
These inequalities prove (\ref{7.8}), (\ref{7.9}) and (\ref{7.11}).
\begin{align}
A(M)^2&\le\int_{M}p_k^\iota d\mu\int_{M}p_k^{-\iota} d\mu, \quad -1\le\iota\le0.\label{7.19}\\
\int_{M}p_k^{-1}d\mu&\le\(\int_{M}p_k^\iota d\mu\)^{-\frac{1}{\iota}}A(M)^{1-\frac{1}{\iota}}, \quad \iota\le-1.\label{7.20}
\end{align}
Combining (\ref{7.17}), (\ref{7.19}) and (\ref{7.20}), we have the proof of (\ref{7.10}).

The proof of (\ref{7.12}) and (\ref{7.13}) are similar to (\ref{7.8}) and (\ref{7.9}).
\end{proof}
Remark: Through Corollary \ref{c7.2}, the quantities in Corollary \ref{c7.3} can be connected with $\int_{M}p_ld\mu$. If we give a bound of $\a$ and $\beta$, $\int_{M}u^\a p_k^\beta d\mu$ may be also  estimated by $\int_{M}p_ld\mu$ through similarly treatment. We will not elaborate here.

\section{Reference}



\begin{biblist}


\bib{B0}{article}{
   author={Andrews B.},
   title={Contraction of convex hypersurfaces in Euclidean space},
   journal={Calc. Var. PDEs},
   volume={2(2)}
   date={1994},
   pages={151-171},
}


\bib{A9}{article}{
   author={Andrews B.},
   title={Contraction of convex hypersurfaces by their affine normal},
   journal={J. Diff. Geom.},
   volume={43}
   date={1996},
   pages={207-230},
}


\bib{B1}{article}{
   author={Andrews B.},
   title={Gauss curvature flow: the fate of the rolling stones},
   journal={Invent. Math.},
   volume={138(1)}
   date={1999},
   pages={151-161},
}

\bib{B3}{article}{
   author={Andrews B.},
   title={Pinching estimates and motion of hypersurfaces by curvature functions},
   journal={J. Reine Angew. Math.},
   volume={608}
   date={2007},
   pages={17-33},
}


\bib{B4}{article}{
   author={Andrews B.},
   author={ McCoy J.},
   author={ Zheng Y.},
   title={Contracting convex hypersurfaces by curvature},
   journal={Calc. Var. PDEs },
   volume={47}
   date={2013},
   pages={611-665},
}

\bib{BLO}{article}{
	author={Barbosa, J. Lucas M.},
	author={Lira, Jorge H. S.},
	author={Oliker, Vladimir I.},
	title={A priori estimates for starshaped compact hypersurfaces with
		prescribed $m$th curvature function in space forms},
	conference={
		title={Nonlinear problems in mathematical physics and related topics,
			I},
	},
	book={
		series={Int. Math. Ser. (N. Y.)},
		volume={1},
		publisher={Kluwer/Plenum, New York},
	},
	date={2002},
	pages={35--52},
	review={\MR{1970603}},
}

\bib{BS}{article}{
   author={Brendle S.},
   author={ Choi K.},
   author={ Daskalopoulos P.},
   title={Asymptotic behavior of flows by powers of the Gauss curvature},
   journal={Acta Math.},
   volume={219(1)}
   date={2017},
   pages={1-16},
   }

\bib{BGL}{article}{
	author={Brendle S.},
	author={Gan P.},
	author={Li J.},
	title={An inverse curvature type hypersurface flow in $\mH^{n+1}$ (\textbf{preprint})},

}

\bib{CB1}{article}{
   author={Chow B.},
   title={Deforming convex hypersurfaces by the $n$-th root of the Gaussian curvature},
   journal={J. Diff. Geom.},
   volume={22(1)}
   date={1985},
   pages={117-138},
}
\bib{CB2}{article}{
   author={Chow B.},
   title={Deforming convex hypersurfaces by the square root of the scalar curvature},
   journal={Invent. Math.},
   volume={87(1)}
   date={1987},
   pages={63-82},
}


\bib{DL}{article}{
	author={Ding S.},
	author={Li G.},
	title={A class of curvature flows expanded by support function and
		curvature function},
	journal={Proc. Amer. Math. Soc.},
	volume={148},
	date={2020},
	number={12},
	pages={5331--5341},
	issn={0002-9939},
	review={\MR{4163845}},
	doi={10.1090/proc/15189},
}


\bib{EH}{article}{
	author={Ecker K.},
	author={Huisken G.},
	title={Immersed hypersurfaces with constant Weingarten curvature},
	journal={Math. Ann.},
	volume={283},
	date={1989},
	number={2},
	pages={329--332},
	issn={0025-5831},
	review={\MR{980601}},
	doi={10.1007/BF01446438},
}


\bib{FWJ}{article}{
  author={Firey W. J.},
     title= {Shapes of worn stones},
 journal={Mathematika},
   volume={21},
     pages={1-11},
     date={1974},
}

\bib{GC5}{article}{
	author={Gerhardt C.},
	title={Curvature flows in the sphere},
	journal={J. Diff. Geom.},
	volume={100},
	date={2015},
	number={2},
	pages={301--347},
	issn={0022-040X},
	review={\MR{3343834}},
}

\bib{GC2}{book}{
	author={Gerhardt C.},
	title={Curvature problems},
	series={Series in Geometry and Topology},
	volume={39},
	publisher={International Press, Somerville, MA},
	date={2006},
	pages={x+323},
	isbn={978-1-57146-162-9},
	isbn={1-57146-162-0},
	review={\MR{2284727}},
}

\bib{GC3}{article}{
	author={Gerhardt C.},
	title={Flow of nonconvex hypersurfaces into spheres},
	journal={J. Diff. Geom.},
	volume={32},
	date={1990},
	number={1},
	pages={299--314},
	issn={0022-040X},
	review={\MR{1064876}},
}

\bib{GC4}{article}{
	author={Gerhardt C.},
	title={Inverse curvature flows in hyperbolic space},
	journal={J. Diff. Geom.},
	volume={89},
	date={2011},
	number={3},
	pages={487--527},
	issn={0022-040X},
	review={\MR{2879249}},
}

\bib{GC}{article}{
   author={Gerhardt C.},
   title={Non-scale-invariant inverse curvature flows in Euclidean space},
   journal={Cal. Var. PDEs},
   volume={49}
   date={2014},
   pages={471-489},
}

\bib{GL}{article}{
	author={Guan P.},
	author={Li J.},
	title={A mean curvature type flow in space forms},
	journal={Int. Math. Res. Not. IMRN},
	date={2015},
	number={13},
	pages={4716--4740},
	issn={1073-7928},
	review={\MR{3439091}},
	doi={10.1093/imrn/rnu081},
}

\bib{GL2}{article}{
	author={Guan P.},
author={Li J.},
	title={The quermassintegral inequalities for $k$-convex starshaped
		domains},
	journal={Adv. Math.},
	volume={221},
	date={2009},
	number={5},
	pages={1725--1732},
	issn={0001-8708},
	review={\MR{2522433}},
	doi={10.1016/j.aim.2009.03.005},
}

\bib{HG}{article}{
   author={Huisken G.},
   title={Flow by mean curvature of convex surfaces into sphere},
   journal={J. Diff. Geom.},
   volume={20(1)}
   date={1984},
   pages={237-266},
}

\bib{HGC}{article}{
   author={Huisken G.},
   author={Sinestrari C.},
   title={Convexity estimates for mean curvature flow and singularities of mean convex surfaces},
   journal={Acta Math.},
   volume={183}
   date={1999},
   pages={45-70},
}

\bib{HLW}{article}{
	author={Hu Y.},
	author={Li H.},
	author={Wei Y.},
	title={Locally constrained curvature flows and geometric inequalities in hyperbolic space},
	journal={Math. Ann.},
	date={2020},
	doi={10.1007/s00208-020-02076-4},
}



\bib{IM2}{article}{
   author={Ivaki M.},
   author={Stancu A.},
   title={Volume preserving centro-affine normal flows},
   journal={Commun. Anal. Geom.},
   volume={21}
   date={2013},
pages={671-685},
}

\bib{IM3}{article}{
   author={Ivaki M.},
   title={Deforming a hypersurface by Gauss curvature and support function},
   journal={J. Funct. Anal.},
   volume={271}
   date={2016},
pages={2133-2165},
}

\bib{IM}{article}{
   author={Ivaki M.},
   title={Deforming a hypersurface by principal radii of curvature and support function},
   journal={Calc. Var. PDEs},
   volume={58(1)}
   date={2019},
}

\bib{JL}{article}{
	author={Jin Q.},
	author={Li Y.},
	title={Starshaped compact hypersurfaces with prescribed $k$-th mean
		curvature in hyperbolic space},
	journal={Discrete Contin. Dyn. Syst.},
	volume={15},
	date={2006},
	number={2},
	pages={367--377},
	issn={1078-0947},
	review={\MR{2199434}},
	doi={10.3934/dcds.2006.15.367},
}

\bib{KK}{article}{
	author={Kwong, Kwok-Kun},
	title={An extension of Hsiung-Minkowski formulas and some applications},
	journal={J. Geom. Anal.},
	volume={26},
	date={2016},
	number={1},
	pages={1--23},
	issn={1050-6926},
	review={\MR{3441501}},
	doi={10.1007/s12220-014-9536-8},
}


\bib{KNV}{book}{
  author={Krylov N. V.},
     title= {Nonlinear elliptic and parabolic quations of the second order},
 publisher={D. Reidel Publishing Co., Dordrecht},
     date={1987. xiv+462pp},

}

\bib{LN}{book}{
  author={Nirenberg L.},
     title= {On a generalization of quasi-conformal mappings and its application to elliptic partial differential equations},
 publisher={Contributions to the theory of partial differential equations, Annals of Mathematics Studies},
     date={ Princeton University Press, Princeton, N. J.,1954, pp. 95C100.}
  }

\bib{LW}{article}{
	author={Li H.},
	author={Wang X.},
	author={Wei Y.},
	title={Surfaces expanding by non-concave curvature functions},
	journal={Ann. Global Anal. Geom.},
	volume={55},
	date={2019},
	number={2},
	pages={243--279},
	issn={0232-704X},
	review={\MR{3923539}},
	doi={10.1007/s10455-018-9625-1},
}

\bib{LSW}{article}{
   author={Li Q.},
   author={Sheng W.},
   author={Wang X-J},
   title={Flow by Gauss curvature to the Aleksandrov and dual Minkowski problems},
   journal={Journal of the European Mathematical Society},
   volume={22}
   date={2019},
   pages={893-923},
}

\bib{SJ2}{article}{
	author={Scheuer J.},
	title={Gradient estimates for inverse curvature flows in hyperbolic
		space},
	journal={Geom. Flows},
	volume={1},
	date={2015},
	number={1},
	pages={11--16},
	review={\MR{3338988}},
	doi={10.1515/geofl-2015-0002},
}

\bib{SJ4}{article}{
	author={Scheuer J.},
	author={Xia C.},
	title={Locally constrained inverse curvature flows},
	journal={Trans. Amer. Math. Soc.},
	volume={372},
	date={2019},
	number={10},
	pages={6771--6803},
	issn={0002-9947},
	review={\MR{4024538}},
	doi={10.1090/tran/7949},
}

\bib{SJ3}{article}{
	author={Scheuer J.},
	title={Non-scale-invariant inverse curvature flows in hyperbolic space},
	journal={Calc. Var. PDEs},
	volume={53},
	date={2015},
	number={1-2},
	pages={91--123},
	issn={0944-2669},
	review={\MR{3336314}},
	doi={10.1007/s00526-014-0742-9},
}

\bib{SJ}{article}{
	author={Scheuer J.},
	title={Pinching and asymptotical roundness for inverse curvature flows in
		Euclidean space},
	journal={J. Geom. Anal.},
	volume={26},
	date={2016},
	number={3},
	pages={2265--2281},
	issn={1050-6926},
	review={\MR{3511477}},
	doi={10.1007/s12220-015-9627-1},
}


\bib{SWM}{article}{
   author={Sheng W.},
   author={Yi C.},
   title={A class of anisotropic expanding curvature flows},
   journal={Discrete and Continuous Dynamical Systems},
   volume={40(4)}
   date={2020},
   pages={2017-2035},
   }

\bib{UJ}{article}{
   author={Urbas J.},
   title={An expansion of convex hypersurfaces},
   journal={J. Diff. Geom.},
   volume={33(1)}
   date={1991},
   pages={91-125},
}

\bib{UJ2}{article}{
	author={Urbas J.},
	title={On the expansion of starshaped hypersurfaces by symmetric
		functions of their principal curvatures},
	journal={Math. Z.},
	volume={205},
	date={1990},
	number={3},
	pages={355--372},
	issn={0025-5874},
	review={\MR{1082861}},
	doi={10.1007/BF02571249},
}



\end{biblist}
\end{document}